\documentclass[a4paper]{amsart}

\usepackage[applemac]{inputenc}
\usepackage[T1]{fontenc}
\usepackage[english]{babel}
\usepackage{babelbib}
\usepackage{enumerate}
\usepackage{hyperref}
\usepackage{amssymb,mathtools}
\usepackage{color}
\usepackage{tkz-fct}
\usepackage{tikz}
\usepackage{tikz-cd}
\usetikzlibrary{matrix}
\usetikzlibrary{arrows}
\usepackage[a4paper]{geometry}

\newcommand{\A}{\mathbf{A}}
\newcommand{\C}{\mathbf{C}}
\newcommand{\D}{\mathbf{D}}

\newcommand{\G}{\mathbf{G}}
\newcommand{\N}{\mathbf{N}}
\renewcommand{\P}{\mathbf{P}}
\newcommand{\Q}{\mathbb{Q}}
\newcommand{\R}{\mathbf{R}}

\newcommand{\Z}{\mathbb{Z}}

\renewcommand{\O}{\cal{O}}

\newcommand{\an}{\textup{an}}

\newcommand{\m}{\mathfrak{m}}

\newcommand{\too}{\longrightarrow}
\newcommand{\mapstoo}{\longmapsto}
\renewcommand{\phi}{\varphi}
\renewcommand{\epsilon}{\varepsilon}
\renewcommand{\ker}{\Ker}

\newcommand{\cal}{\mathcal}

\DeclareMathOperator{\pr}{pr}

\DeclareMathOperator{\Spec}{Spec}

\DeclareMathOperator{\Hom}{Hom}

\DeclareMathOperator{\Ker}{Ker}

\DeclareMathOperator{\coker}{Coker}

\DeclareMathOperator{\Supp}{Supp}

\DeclareMathOperator{\Ann}{Ann}

\renewcommand{\le}{\leqslant}
\renewcommand{\ge}{\geqslant}

\newcommand{\E}[2]{\ensuremath{\mathbf{A}^{#1,\mathrm{an}}_{#2}}}

\newcommand{\Ahyb}[2]{\ensuremath{\mathbf{A}^{#1,\mathrm{hyb}}_{#2}}}

\newcommand{\Uk}{\mathfrak{U}}
\newcommand{\Xk}{\mathfrak{X}}
\newcommand{\Yk}{\mathfrak{Y}}

\newcommand{\cA}{\mathcal{A}}

\newcommand{\cF}{\mathcal{F}}

\newcommand{\cG}{\mathcal{G}}

\newcommand{\cH}{\mathcal{H}}

\newcommand{\cI}{\mathcal{I}}
\newcommand{\cJ}{\mathcal{J}}

\renewcommand{\Mc}{\mathcal{M}}
\newcommand{\cM}{\mathcal{M}}

\newcommand{\cO}{\mathcal{O}}

\newcommand{\Vc}{\mathcal{V}}

\newcommand{\Xc}{\mathcal{X}}
\newcommand{\eps}{\varepsilon}

\DeclarePairedDelimiter\abs{\lvert}{\rvert}

\DeclarePairedDelimiter\norm{\lVert}{\rVert}

\newcommand\wc{{\mkern 2mu\cdot\mkern 2mu}}
\newcommand\va{\abs{\wc}}
\newcommand\nm{\norm{\wc}}

\newcommand\cp{^{\!\urcorner}}

\newcommand\hyb{\mathrm{hyb}}

\newcommand{\simto}{\xrightarrow[]{\sim}}
\DeclareMathOperator{\colim}{colim}

\newcommand{\binphi}{\mathbin{\Phi}}

\DeclareFontFamily{U}{mathc}{}
\DeclareFontShape{U}{mathc}{m}{it}%
{<->s*[1.03] mathc10}{}
\DeclareMathAlphabet{\mathcal}{U}{mathc}{m}{it}
\DeclareMathOperator{\sHom}{\mathcal{H\mkern-3mu om}}
\DeclareMathOperator{\sKer}{\mathcal{K\mkern-3mu er}}

\theoremstyle{plain}
\newtheorem{theorem}{Theorem}[section]
\newtheorem{lemma}[theorem]{Lemma}
\newtheorem{proposition}[theorem]{Proposition}
\newtheorem{corollary}[theorem]{Corollary}

\newtheorem{theoremintro}{Theorem}

\theoremstyle{definition}
\newtheorem{definition}[theorem]{Definition}
\newtheorem{example}[theorem]{Example}
\newtheorem{notation}[theorem]{Notation}

\theoremstyle{remark}
\newtheorem{remark}[theorem]{Remark}

\numberwithin{equation}{section}

\begin{document}
\title{Valuative compactifications of analytic varieties}

\author{J\'er\^ome Poineau}
\address{Normandie Univ., UNICAEN, CNRS, Laboratoire de math\'ematiques Nicolas Oresme, 14000 Caen, France}
\email{\href{mailto:jerome.poineau@unicaen.fr}{jerome.poineau@unicaen.fr}}
\urladdr{\url{https://poineau.users.lmno.cnrs.fr/}}

\date{\today}


\begin{abstract}
Let $X$ be an algebraic variety over~$\C$. We define a canonical compactification~$X\cp$ of the complex analytic space~$X(\C)$ by adding a Berkovich space over a trivially valued field at the boundary. The construction is functorial with respect to proper morphisms and preserves many properties, such as normality, regularity, etc. We prove a partial GAGA theorem in this setting~: there is an equivalence between the categories of coherent sheaves on~$X$ and~$X\cp$, and it induces bijections on global sections. The results still hold if~$\C$ is replaced by a complete non-trivially valued field~$k$, and complex analytic spaces by Berkovich analytic spaces over~$k$.
\end{abstract}

\maketitle

Let $X$ be an algebraic variety over the field of complex numbers~$\C$, and consider its analytification $X^\an$, as a locally ringed space (in other words, $X(\C)$ endowed with the transcendental topology and the sheaf of analytic functions). 
The goal this article is to define a canonical compactification~$X\cp$ of~$X^\an$, namely a compact locally ringed space~$X\cp$ in which $X^\an$ embeds as a dense open, and study its properties.

Within the realm of complex analytic spaces, this task is hopeless. Indeed, a space as simple as the affine plane~$\C^2$ may already be embedded into different natural compact spaces, such as $\P^2(\C)$ or $\P^1(\C) \times \P^1(\C)$, which may or may not be relevant, depending on the issue at hand.

Former research however suggests that canonical compactifications may exist, in some larger category of spaces. 
During the eighties, J.~Morgan and P.~Shalen carried out their seminal work on fundamental groups of 3-manifolds \cite{MSI} by studying compactifications of character varieties. The latter are defined by adding ``ideal points'', defined in terms of valuations, at the boundaries of the varieties. 

More recently, a similar strategy has been used by various authors to investigate degenerations of complex analytic varieties, based on the idea that the limit behavior of a one-parameter family (say over a punctured disc with parameter $t$) may be encoded into a non-Archimedean Berkovich analytic space (defined over the non-Archimedean field of Laurent series $\C(\!(t)\!)$). This seemingly simple principle has been intensively used over the last years, and striking applications have been found to topics as diverse as Hodge theory (limit mixed Hodge structures \cite{BerkovichW0}), complex manifolds (degenerations of volume forms \cite{BJ}), dynamical systems (degeneration of measures \cite{FavreEndomorphisms}, blow-ups of Lyapunov exponents \cite{DujardinFavreSL2C}), number theory (uniform Manin-Mumford bounds \cite{DKY}), etc.

Note that degenerations may be interpreted as compactifications in a relative setting (over a punctured disc for instance), so that the former may be obtained from the latter, under suitable functoriality properties.

In this text, we propose to lay the foundation of a thorough theory of compactifications (and degenerations) of complex analytic varieties encompassing all preceding instances, by relying on the theory of Berkovich spaces over Banach rings. Along the way, this flexible setting will allow us to enrich the compactifications with new additional features, such as a locally ringed space structure.

Our main result may be summarized as follows.

\begin{theoremintro}[\textit{infra} Theorem~\ref{th:compact}, Propositions~\ref{prop:functorcp} and \ref{prop:iotaiso}, Corollaries~\ref{cor:propertiesXcp} and~\ref{cor:propertiesfcp}]\label{thintro:compactification}
To each complex algebraic variety~$X$, we may associate a locally ringed space~$X\cp$ that satisfies the following properties.
\begin{enumerate}[i)]
\item The space $X\cp$ is compact. 
\item There exists a natural morphism of locally ringed spaces $X^\an \to X\cp$. It is an open immersion with dense image.
\end{enumerate}

The association $X \mapsto X\cp$ is functorial with respect to proper morphisms, and preserves many properties: normality, regularity, etc.
\end{theoremintro}

\medbreak 

We investigate coherent sheaves on compactifications and obtain a partial GAGA theorem.

\begin{theoremintro}[\textit{infra} Theorems~\ref{th:H0} and~\ref{th:FF+}]\label{thintro:GAGA}
Let $X$ be a complex algebraic variety. To each coherent sheaf~$F$ on~$X$, we may functorially associate a coherent sheaf~$F\cp$ on~$X\cp$.
\begin{enumerate}[i)]
\item The functor 
\[F \in \mathrm{Coh}(X) \longmapsto F\cp \in \mathrm{Coh}(X\cp)\]
is an equivalence of categories.
\item For each coherent sheaf $F$ on~$X$, we have a natural isomorphism
\[H^0(X,F) \stackrel{\sim}{\longrightarrow} H^0(X\cp,F\cp).\]
\end{enumerate}
\end{theoremintro}

Both results still hold if $X$ is an algebraic variety over a complete non-Archimedean valued field~$k$ with non-trivial valuation. In this case, $X^\an$ is to be understood as the Berkovich analytification of~$X$ over~$k$. 

\bigbreak

Let us now say a few words about the definition of the compactification. The missing part $X\cp - X^\an$, in other words the boundary of~$X^\an$, already appears in the work of V.~Berkovich \cite{BerkovichtoDrinfeld}, O.~Ben--Bassat and M.~Temkin \cite{BBT}, and A.~Thuillier \cite{beth}. Its construction takes place in the setting of Berkovich spaces over the trivially valued field~$\C$ \cite{rouge,bleu}. To combine it with the complex analytic space~$X^\an$, we work in the setting of hybrid Berkovich spaces, a particular instance of Berkovich spaces over Banach rings, as developed in \cite{rouge,CTCZ}. A quotient operation, inspired by L.~Fantini's definition of normalized spaces \cite{TXZ}, completes the construction.

During the last stage of the writing of this article, we discovered that Yinchong Song had introduced the same compactification procedure in~\cite{SongDivisors} (from the topological point of view only) to study adelic divisors over quasi-projective varieties in the sense of X.~Yuan and S.-W.~Zhang \cite{YuanZhangAdelic}. We thank Mattias Jonsson for pointing out this reference.

\medbreak

\noindent $\blacktriangleright$ \textbf{The boundary}

The construction of the boundary of~$X^\an$ takes place in the category of Berkovich analytic spaces over~$(\C,\va_{0})$, where $\va_{0}$ denotes the trivial absolute value. Recall that we have an analytification functor $X \mapsto X^\an_{0}$ from algebraic varieties over~$\C$ to analytic spaces over~$(\C,\va_{0})$. 

Endowing~$\C$ with the discrete topology, we may consider the scheme~$X$ as a formal scheme, and consider its generic fiber~$X^\beth$ in the sense of Raynaud. 
It is a compact subset of~$X_{0}^\an$, and we may now define 
\[X_{\infty} := X_{0}^\an - X^\beth,\]
which corresponds to the set of seminorms without center on~$X$. 

It is worth noticing that, if~$X$ is embedded as an open subset in a proper algebraic variety~$Y$ over~$\C$ with complement~$Z$, then $X_{\infty}$ may be identified with the generic fiber (in the sense of Raynaud--Berthelot) of the formal completion~$\hat Y_{Z}$ of~$Y$ along~$Z$ deprived of (the analytication of) its special fiber~$Z$. In particular, the latter construction does not depend on the choice of~$Y$.

The set~$X_{\infty}$ was first defined by V.~Berkovich in a letter to V.~Drinfeld \cite{BerkovichtoDrinfeld}, and subsequentely used by O.~Ben--Bassat and M.~Temkin in~\cite{BBT} to prove some descent results (reconstructing coherent sheaves on~$Y$ from coherent sheaves on~$\hat Y_{Z}$ and~$X$). It was also independently defined by A.~Thuillier in~\cite{beth}, who proved that if~$Y$ is regular and~$Z$ has normal crossings, then the dual complex of~$Z$ is homotopy equivalent to~$X_{\infty}$. As a consequence, the homotopy type of the dual complex of the boundary depends only on~$X$ and not on the chosen compactification.

\medbreak

\noindent $\blacktriangleright$ \textbf{Hybrid Berkovich spaces}

It remains to put together the analytic space~$X^\an$ and its boundary~$X_{\infty}$. For this purpose, we use the framework of Berkovich spaces over Banach rings, as introduced in~\cite{rouge} and later developped in \cite{A1Z,EtudeLocale,CTCZ}.  

Consider the field ~$\C$ endowed with the hybrid norm $\nm_{\hyb} := \max(\va_{0},\va_{\infty})$, where~$\va_{\infty}$ is the usual absolute value on~$\C$. It is a Banach ring, and we may consider the analytification~$X^\hyb$ of~$X$ over it.  The basic example is the analytification of $\Spec(\C)$, which may be explicitly described as 
\[\Spec(\C)^\hyb = \{\va_{\infty}^\eps, 0\le \eps \le 1\},\] 
where $\va_{\infty}^0 := \va_{0}$. 

The structure morphism $\pr \colon X \to \Spec(\C)$ gives rise to a morphism $\pr^\hyb \colon X^\hyb \to \Spec(\C)^\hyb$, by functoriality. Its fiber over~$\va_{0}$ coincides with~$X^\an_{0}$, whereas the fibers over the different~$\va_{\infty}^\eps$, with $\eps \in (0,1]$, all identifies to~$X^\an$.\footnote{The parameter~$\eps$ plays a role as a scaling factor for the metric, but no more at the topological level.} It makes sense to interpret the morphism~$\pr^\hyb$ as a degeneration of a family of complex analytic spaces onto a non-Archimedean Berkovich space.

\medbreak

\noindent $\blacktriangleright$ \textbf{The compactification}

Set
\[X^+ := X^\hyb - X^\beth.\]
It is an open subset of~$X^\hyb$, whence inherits a structure of locally ringed space. The space~$X^+$ is very similar to~$X^\hyb$. It has the same Archimedean part: a family of spaces~$X^\an$ parametrized by~$(0,1]$, and only differs in its non-Archimedean part, with~$X_{\infty}$ replacing~$X_{0}^\an$.

In order to keep only one copy of~$X^\an$, we identify the points of the space~$X^+$ that correspond to equivalent seminorms. The resulting space~$X\cp$, endowed with the push-forward structure sheaf, is the promised compactification.

By construction, the Archimedean part of the space~$X\cp$ now consists in exactly one copy of~$X^\an$. Its non-Archimedean part, which is the quotient of~$X_{\infty}$ by the equivalence of seminorms, is a so-called normalized space, as introduced by L.~Fantini in~\cite{TXZ}.

\medbreak

\begin{center}
\textbf{Organization of the manuscript}
\end{center}

The first two sections gather general material about Berkovich spaces over Banach rings. In Section~\ref{sec:BerkovichBanach}, we briefly recall their definition. In Section~\ref{sec:hybridspaces}, we focus on the case of hybrid Berkovich spaces, that is to say spaces over Banach rings of the form $(\C,\nm_{\hyb})$, possibly allowing other valued fields instead of~$\C$. We carefully investigate properties of the \emph{flow}, which raises multiplicative seminorms (corresponding to points of Berkovich spaces) to some power. The latter plays a fundamental role in the last part of the construction of the compactification, where the quotient is performed.

The rest of the paper deals more specifically with the compactification problem. In Section~\ref{sec:triviallyvalued}, we recall the construction of the boundary~$X_{\infty}$ and its main properties.
In Section~\ref{sec:compactifications}, we perform the compactification procedure by defining~$X^+$ and~$X\cp$. We investigate their topological properties as well as their first sheaf-theoretic properties, and prove Theorem~\ref{thintro:compactification}.
Finally, in Section~\ref{sec:GAGA}, we investigate coherent sheaves on~$X^+$, and prove Theorem~\ref{thintro:GAGA}. The general strategy is to reduce to a classical GAGA theorem on some fixed compactification~$Y$ of~$X$, by means of extension results for coherent sheaves from~$X^+$ to~$Y^\an$. The main technical input comes from the Stein properties (higher cohomological vanishing and global generation for coherent sheaves) of hybrid discs from \cite{CTCZ}.

\medbreak

\begin{center}
\textbf{Conventions}
\end{center}

We denote by~$\N$ (resp. $\N_{\ge 1}$) the set of non-negative (resp. positive) integers.

A \emph{variety} over a field~$k$ is a separated scheme of finite type over~$k$ (hence quasi-compact).

\medbreak

\begin{center}
\textbf{Acknowledgements}
\end{center}

We express our heartfelt gratitude to Marco Maculan for early discussions that started the project, and many more that contributed to his achievement. 
We also thank Antoine Ducros, Mattias Jonsson and Alexandre Roy for several comments.

\section{Berkovich analytic spaces over Banach rings}\label{sec:BerkovichBanach}

Let $(\cA,\nm)$ be a Banach ring. In this section, we recall V.~Berkovich's definition of $\cA$-analytic space (see \cite[Section~1.5]{rouge}).

\subsection{Definitions}

We start with the affine analytic space~$\E{n}{\cA}$ of dimension $n$ over~$\cA$, for $n \in \N$. It bears the structure of a locally ringed space and we define it in three steps: set, topology and structure sheaf. 

\medbreak

\noindent \textit{Set:} The underlying set of $\E{n}{\cA}$ is the set of bounded multiplicative seminorms on $\cA[T_{1},\dotsc,T_{n}]$ that are bounded on~$\cA$, \emph{i.e.} the set of maps 
\[ \va \colon \cA[T_{1},\dotsc,T_{n}] \too \R_{\ge 0} \]
that satisfy the following properties:
\begin{enumerate}[i)]
\item $|0|=0$ and $|1|=1$;
\item $\forall P,Q \in \cA[T_{1},\dotsc,T_{n}]$, $|P+Q| \le |P| + |Q|$;
\item $\forall P,Q \in \cA[T_{1},\dotsc,T_{n}]$, $|P Q| = |P|\, |Q|$;
\item $\forall a \in \cA$, $|a| \le \|a\|$.
\end{enumerate}

We set $\cM(\cA) := \E{0}{\cA}$ and call it the \emph{spectrum} of~$\cA$. Note that we have a projection map $\pr \colon \E{n}{\cA} \to \cM(\cA)$ induced by the morphism $\cA \to \cA[T_{1},\dotsc,T_{n}]$.

Let~$x$ be a point of~$\E{n}{\cA}$. Denote by~$\va_{x}$ the multiplicative seminorm associated to it. The ring $\cA[T_{1},\dotsc,T_{n}]/\ker(\va_{x})$ is a domain and we can consider its field of fractions. The seminorm~$\va_{x}$ induces an absolute value on the later, and we can consider its completion, which we denote by~$\cH(x)$. We simply denote by~$\va$ the absolute value on~$\cH(x)$ induced by~$\va_{x}$, since no confusion may result.

We have a natural morphism $\chi_{x}\colon \cA[T_{1},\dotsc,T_{n}] \to \cH(x)$. For each $P \in \cA[T_{1},\dotsc,T_{n}]$, we set $P(x) := \chi_{x}(P)$. Note that, by definition, we have $|P(x)| = |P|_{x}$.

\medbreak

\noindent \textit{Topology:} The set $\E{n}{\cA}$ is endowed with the coarsest topology such that, for each $P \in \cA[T_{1},\dotsc,T_{n}]$, the map
\[x \in \E{n}{\cA} \mapstoo |P(x)| \in \R_{\ge0}\] 
is continuous. The resulting topological space is Hausdorff and locally compact. The spectrum~$\cM(\cA)$ is even compact. The projection map~$\pr$ is continuous.

For each $x\in \cM(\cA)$, we have a natural homeomorphism
\[\E{n}{\cH(x)} \simto \pr^{-1}(x).\] 

\medbreak

\noindent \textit{Structure sheaf:} For each open subset~$V$ of~$\E{n}{\cA}$, we denote by~$S_{V}$ the set of element of $\cA[T_{1},\dotsc,T_{n}]$ that do not vanish on~$V$ and set $K(V) := S_{V}^{-1} \cA[T_{1},\dotsc,T_{n}]$. 

Let~$U$ be an open subset of~$\E{n}{\cA}$. We define~$\cO(U)$ to be the set of maps
\[f \colon U \too \bigsqcup_{x\in U} \cH(x)\]
such that
\begin{enumerate}[i)]
\item for each $x\in U$, $f(x)\in \cH(x)$;
\item each $x\in U$ has an open neighborhood~$V$ on which~$f$ is a uniform limit of elements of~$K(V)$. 
\end{enumerate}

\medbreak

One may now define arbitrary \emph{$\cA$-analytic spaces}, by first considering \emph{local models}, \textit{i.e.} closed analytic subspaces of open subsets of affine spaces, and then gluing the latter. Note that those spaces are not merely locally ringed spaces, since one wants to keep track of some metric information on the rings of sections, for instance. The actual definition thus involves several pieces of data such as atlases, embeddings into affine spaces for local models, etc. We refer to \cite[2.1]{CTCZ} for details.

Each $\cA$-analytic space~$X$ admits a canonical projection map $\pr \colon X \to \cM(\cA)$. Moreover, for each $x\in \cM(\cA)$, the fiber $\pr^{-1}(x)$ may naturally be endowed with a structure of $\cH(x)$-analytic space.

\medbreak

In the rest of the text, we will consider sections of sheaves on arbitrary subsets of a space. Recall that, given a space~$S$ and a sheaf~$\cF$ on it, for each subset~$T$ of~$S$, one may define the set of \textit{overconvergent} sections
\[\cF(T) := \colim_{U \supseteq T} \cF(U),\]
where $U$ runs through the open neighborhoods of~$T$ in~$S$. Note that it gives back the usual $\cF(T)$ when $T$ is open.

In the case where~$(S,\cO)$ is a locally ringed space, by the former procedure, any subset~$T$ inherits a locally ringed space structure, called overconvergent. The structure sheaf on~$T$ identifies to $\iota^{-1}\cO$, where $\iota\colon T \to S$ in the inclusion. We usually still denote it by~$\cO$.

As a consequence of the preceding discussion, the notion of coherent sheaf on an arbitrary subset of an analytic space makes sense. Let us quote the following useful result (see \cite[Proposition~1]{Cartan5152exp19}). The reference provides a proof for complex analytic spaces, but it adapts readily to a more general setting. 

\begin{proposition}\label{prop:extensioncoherentsheafcompact}
Let $S$ be an $\cA$-analytic space and~$K$ be a compact subset of~$S$. Each coherent sheaf on~$K$ extends to a coherent sheaf on a neighborhood of~$K$.
\qed
\end{proposition}

\subsection{Base rings}

The theory of Berkovich spaces over~$\cA$ has been extensively studied in the case where $\cA$ is a complete non-Archimedean valued field (see \cite{rouge, bleu} for instance). The case where~$\cA$ is an arbitrary Banach ring has comparatively received little attention. 

In recent work, the author and his collaborators have tried to single out properties of the ring~$\cA$ that are sufficient to ensure the existence of a well-behaved theory of $\cA$-analytic spaces (see \cite{A1Z,EtudeLocale,CTCZ}). This led to the notion of \emph{geometric base ring}  (see \cite[D\'efinition~3.3.8]{CTCZ}). Over such a ring, several expected basic properties of the spaces are known to hold: Noetherianity and Henselianity of the local rings, coherence of the structure sheaf, etc. One may also define a notion of morphism of $\cA$-analytic spaces, and the resulting category is stable under various operations, such as fiber products.

The precise conditions for~$\cA$ to be a geometric base ring are long, technical and not particularly illuminating. Instead of stating them, we provide a list of examples where they are satisfied and that already yield interesting applications:
\begin{enumerate}[i)]
\item a non-Archimedean complete valued field $(k,\va)$, which gives back a large part of the usual theory of $k$-analytic spaces (good spaces with no boundary in Berkovich's terminology);
\item an Archimedean complete valued field $(k,\va)$, including the case of $(\C,\va_{\infty})$, where $\va_{\infty}$ denotes the usual absolute value, which gives back the usual theory of complex analytic spaces;
\item a complete discrete valuation ring $(R,\va)$, where $\va$ is associated to the valuation;
\item the ring of integers $(\Z,\va_{\infty})$ or the ring of integers of a number field. 
\end{enumerate}

In this paper, we will consider another example, that of hybrid fields. We defer the definition to Section~\ref{sec:defhybrid}.

\subsection{Analytification}\label{sec:analytification}

Assume that $\cA$ is a geometric base ring. In this section, we recall how to associate naturally an $\cA$-analytic space to a scheme over~$\cA$, under suitable assumptions. We denote by $\Hom_{\cA - \textrm{loc}}(\wc,\wc)$ the set of morphisms in the category of locally $\cA$-ringed spaces. 

Let $X$ be a scheme locally of finite presentation over~$\cA$. By \cite[Th\'eor\`eme~4.1.4]{CTCZ}, the functor
\[\begin{array}{ccc}
\cA - \textrm{An} & \too & \textrm{Set}\\
S & \mapsto &\Hom_{\cA - \textrm{loc}}(S,X)
\end{array}\]
is representable. The analytic space representing the functor is called \emph{analytification of~$X$}, and denoted by~$X^\an$. It comes with a morphism of locally $\cA$-ringed spaces
\[\rho_{X} \colon X^\an \too X.\]

\medbreak

It follows from the proof of \cite[Th\'eor\`eme~4.1.4]{CTCZ} that the analytification~$X^\an$ may be constructed by the same procedure as for complex analytic spaces. We recall it here.

\medbreak

\noindent \textit{Step~1:} $X$ is of the form $\A^n_{\cA}$, for some $n\in \N$.

The analytification is $X^\an = \E{n}{\cA}$ and the morphism
\[\rho_{\A^n_{\cA}} \colon \E{n}{\cA} \too \A^n_{\cA}\]
is the one sending a multiplicative seminorm to its kernel.

\medbreak

\noindent \textit{Step~2:} $X$ is an affine variety of finite presentation over~$\cA$.

The variety~$X$ is a closed subscheme of some affine space~$\A^n_{\cA}$ defined by a finitely generated ideal~$I$ of~$\cO(\A^n_{\cA})$.
Its analytification~$X^\an$ is the closed analytic subspace of~$\E{n}{\cA}$ defined by the sheaf of ideals generated by~$I$ and the morphism $\rho_{X}$ is the one induced by~$\rho_{\A^n_{\cA}}$.

\medbreak

\noindent \textit{Step~3:} $X$ is locally of finite presentation over~$\cA$.

The variety~$X$ admits an open covering~$(U_{i})_{i\in I}$ by affine varieties of finite presentation over~$\cA$. Its analytification~$X^\an$ and the morphism~$\rho_{X}$ are obtained by gluing the~$U_{i}^\an$'s and the $\rho_{U_{i}}$'s constructed previously.

\begin{remark}\label{rem:pairs}
From the explicit construction of the analytification, we may derive an alternative useful description of the underlying set of $X^\an$. It may be identified with the set of pairs
\[ \{ (q,\va) : q\in X,\ \va \textrm{ absolute value on } \kappa(q) \textrm{ such that } \va_{\vert \cA} \le \nm\}.\]
\end{remark}

\section{Hybrid Berkovich spaces}\label{sec:hybridspaces}

In this section, we focus on the case of \emph{hybrid spaces}, that is to say analytic spaces over a hybrid field.

\subsection{Definitions}\label{sec:defhybrid}

Let $k$ be a field endowed with a non-trivial absolute value~$\va$ (no completeness condition required).

\begin{definition}\label{def:khyb}
We call \emph{trivial absolute value} on~$k$ the absolute value defined by
\[ \va_{0} \colon a \in k \mapstoo 
\begin{cases}
0 & \textrm{if } a=0;\\
1 & \textrm{otherwise}
\end{cases}.\]
We call \emph{hybrid norm} on~$k$ the norm defined by
\[\nm_{\hyb} := \max(\va_{0},\va).\]
The pair $(k,\nm_{\hyb})$ is a Banach ring called \emph{hybrid field} associated to~$k$. We denote it by~$k_\hyb$.
\end{definition}

\begin{notation}
We denote by $k_{0}$ the field~$k$ endowed with the trivial absolute value~$\va_{0}$.

We denote by $\hat k$ the completion of the field~$k$ with respect to~$\va$. For $\eps\in (0,1]$, we denote by~$\hat k_{\eps}$ the completion of the field~$k$ with respect to~$\va_{\eps}$. Note that all these fields are isomorphic.
\end{notation}

The spectrum of the hybrid field~$k_\hyb$ may be easily described: the map
\[\eps \in [0,1] \mapstoo \va^\eps \in \cM(k_\hyb)\]
is a homeomorphism. In the sequel, we will identity $\cM(k_\hyb)$ and $[0,1]$ by the map above. With this notation, we have $\cH(0) = k_{0}$ and, for each $\eps\in(0,1]$, $\cH(\eps)=\hat k_{\eps}$.

\begin{notation}\label{nota:Xeps}
Let $X$ be a $k_\hyb$-analytic space, with projection map $\pr \colon X \to \cM(k_\hyb) = [0,1]$. 

For each $\eps \in [0,1]$, set $X_{\eps} := \pr^{-1}(\eps)$. It is naturally an $\cH(\eps)$-analytic space.

Set $X_{>0} := X - X_{0}$.
\end{notation}

\begin{remark}\label{rem:kkeps}
Let $X$ be a $k_\hyb$-analytic space. All the spaces~$X_{\eps}$, with $\eps \in (0,1]$, are of a similar flavour. Indeed, a $\hat k_{\eps}$-analytic space is essentially a $\hat k$-analytic space with a different normalization, and the underlying locally ringed spaces are isomorphic. On the other hand, the space $X_{0}$ is a $k_{0}$-analytic space, hence a space over a trivially valued field, rather different from the former. 

The case where $(k,\va) = (\C,\va_{\infty})$ is particularly striking: the fibers over $(0,1]$ are complex analytic spaces, or closely related spaces, whereas the fiber over~0 is a non-Archimedean Berkovich space. One may interpret this picture as a family of complex analytic spaces degenerating onto a non-Archimedean space. This point of view gave rise to many applications, such as \cite{BerkovichW0},\cite{BJ},\cite{FavreEndomorphisms},\cite{DujardinFavreSL2C},\cite{DKY}, etc.
\end{remark}

\subsection{The flow}\label{sec:flow}

Let $k$ be a field endowed with a non-trivial absolute value~$\va$.

\subsubsection{Definitions and first properties}

\begin{definition}
For $\eps \in [0,1]$, set 
\[I_{\eps} := 
\begin{cases}
[0,+\infty) &\textrm{if } \eps =0;\\
[0,\frac1\eps] &\textrm{if } \eps >0.
\end{cases}\]
and $I^\ast_{\eps} = I_{\eps} -\{0\}$.

Let~$S$ be a $k_\hyb$-analytic space and consider the projection map $\pr \colon S \to \Mc(k_\hyb) = [0,1]$. For each $x\in S$, we set $I_{x} := I_{\pr(x)}$ and $I_{x}^\ast := I^\ast_{\pr(x)}$.
\end{definition}

\begin{lemma}\label{lem:xalpha}
Let $x\in \A^{n,\hyb}_{k}$ and $\alpha\in I_{x}$. The map
\[P \in k[T_{1},\dotsc,T_{n}] \mapsto |P(x)|^\alpha \in \R_{\ge 0}\]
is a multiplicative seminorm on~$k[T_{1},\dotsc,T_{n}]$ that is bounded by~$\nm_{\hyb}$ on~$k$. In other terms, it defines a point in $\A^{n,\hyb}_{k}$.
\end{lemma}
\begin{proof}
By definition, the map $P \in k[T_{1},\dotsc,T_{n}] \mapsto |P(x)| \in \R_{\ge 0}$ is a multiplicative seminorm on~$k[T_{1},\dotsc,T_{n}]$ that is bounded by~$\nm_{\hyb}$ on~$k$. It follows that the map $P \in k[T_{1},\dotsc,T_{n}] \mapsto |P(x)|^\alpha \in \R_{\ge 0}$ satisfies the same properties. The only non-trivial part is the fact that the triangle inequality still holds, in the case where $k$ is Archimedean. This is ensured by the choice of~$\alpha$ (see for instance \cite[VI, \S 6, n$^\circ$1, Proposition~2]{BourbakiAC56}). 
\end{proof}

\begin{notation}
Let $x\in \A^{n,\hyb}_{k}$ and $\alpha\in I_{x}$. We denote by $x^\alpha \in  \A^{n,\hyb}_{k}$ the point associated to the map $P \mapsto |P(x)|^\alpha$ from Lemma~\ref{lem:xalpha}. 

If $\alpha\in I_{x}^\ast$, the fields $\cH(x)$ and $\cH(x^\alpha)$ are obtained by completing the same field for equivalent absolute values, which yields an isomorphism
\[ \lambda_{x^\alpha,x} \colon \cH(x) \xrightarrow[]{\sim} \cH(x^\alpha).\]
\end{notation}

\begin{remark}
We have $\pr(x^\alpha) = \alpha\, \pr(x)$.

The isomorphism $\lambda_{x^\alpha,x}$ is not an isometry, but raises the absolute value to the power~$\alpha$:
for each $u \in \cH(x)$, we have
\[ \abs{\lambda_{x^\alpha,x}(u)} = \abs{u}^\alpha.\]
In particular, 
for each $P \in k[T_{1},\dotsc,T_{n}]$, we have
\[ \abs{\lambda_{x^\alpha,x}(P)(x^\alpha)} = \abs{P(x)}^\alpha.\]
\end{remark}

Although it is possible to define $x^\alpha$ for each $\alpha \in I_{x}$, in the case $\alpha=0$, several properties are lost. In the sequel, we will only consider $\alpha \in I_{x}^\ast$.

\begin{definition}
Set
\[D(\Ahyb{n}{k}) := \bigcup_{x\in \Ahyb{n}{k}} \{x\} \times I^\ast_{x} \subseteq \A^{n,\hyb}_{k} \times \R_{>0}.\]
We define the \emph{flow} on~$\Ahyb{n}{k}$ as the map
\[ 
\begin{array}{ccccc}
\Phi & \colon & D(\Ahyb{n}{k}) & \too &  \A^{n,\hyb}_{k}\\
&& (x,\alpha) & \mapstoo & x^\alpha
\end{array}.
\]
\end{definition}

\begin{proposition}\label{prop:Phicontinuous}
The flow $\Phi \colon D(\Ahyb{n}{k}) \to \Ahyb{n}{k}$ is continuous and open.
\end{proposition}
\begin{proof}
Continuity follows from the definitions (see the proof of \cite[Proposition~1.3.4]{A1Z} for details). 

Let us prove openness. Let $(x,\alpha) \in D(\Ahyb{n}{k})$ and let $W$ be a neighborhood of $(x,\alpha)$ in $D(\Ahyb{n}{k})$. It is enough to prove that $\Phi(W)$ contains a neighborhood of~$x^\alpha$ in~$\Ahyb{n}{k}$.

By definition of the topology of $\Ahyb{n}{k}$, there exists $m \in \N$, $P_{1},\dotsc,P_{m} \in k[T_{1},\dotsc,T_{n}]$, $r_{1},\dotsc,r_{m},s_{1},\dotsc,s_{m} \in \R$ with $r_{i} < \abs{P_{i}(x)} < s_{i}$ for all $i = 1,\dotsc, m$ and $\beta,\gamma \in \R_{>0}$ with $\beta < \alpha < \gamma$ such that 
\[ W' := \big(\{y \in \Ahyb{n}{k} : \forall i=1,\dotsc,m, \ r_{i} < \abs{P_{i}(y)} < s_{i} \} \times (\beta,\gamma) \big) \cap D(\Ahyb{n}{k})\]
is an open neighborhood of~$(x,\alpha)$ contained in~$W$. 

\medbreak

\noindent $\bullet$ Assume that $\alpha \ge 1$.

Set 
\[ U := \{z \in \Ahyb{n}{k} : \forall i=1,\dotsc,m, \ r_{i}^\alpha < \abs{P_{i}(z)} < s_{i}^\alpha \}.\]
It is an open neighborhood of~$x^\alpha$ in~$\Ahyb{n}{k}$. 

For each $z \in U$, we have $\frac1{\alpha} \in I_{z}^\ast$, and $(z^{1/\alpha},\alpha) \in W'$. It follows that $\Phi(W') \supseteq U$.

\medbreak

\noindent $\bullet$ Assume that $\alpha < 1$.

Note that, we have $p(x^\alpha) \le \alpha$. Let $\alpha' \in (\alpha,\min(\gamma,1))$ such that, for each $i=1,\dotsc,m$, we have $r_{i}^{\alpha'/\alpha} < \abs{P_{i}(x)} < s_{i}^{\alpha'/\alpha}$. Then, the set 
\[ V := \{z \in \Ahyb{n}{k} : \forall i=1,\dotsc,m, \ r_{i}^{\alpha'} < \abs{P_{i}(z)} < s_{i}^{\alpha'} \} \cap \pr^{-1}([0,\alpha'))\]
is an open neighborhood of~$x^\alpha$ in~$\Ahyb{n}{k}$. 

For each $z \in V$, we have $\frac1{\alpha'} \in I_{z}^\ast$, and $(z^{1/\alpha'},\alpha') \in W'$. It follows that $\Phi(W') \supseteq V$.
\end{proof}

\begin{definition}\label{def:T(x)} 
Let $x\in \A^{n,\hyb}_{k}$. The \emph{trajectory} of~$x$ is the set
\[T(x) := \Phi(x,I_{x}^\ast) = \{x^\alpha \colon \alpha\in I^\ast_{x}\} \subseteq \Ahyb{n}{k}.\]
\end{definition}

\begin{remark}\label{rem:shapeT(x)}
The trajectory of a point may have two different shapes. Let $x\in \A^{n,\hyb}_{k}$. 

\noindent $\bullet$ Assume that $\cH(x)$ is trivially valued. 

Then, $I_{x}^\ast = (0,+\infty)$ and, for each $\alpha \in (0,+\infty)$, we have $x^\alpha=x$. In particular, $T(x) = \{x\}$.

\noindent $\bullet$ Assume that $\cH(x)$ is not trivially valued. 

In this case, the map $\Phi(x,\wc)$ is injective, and provides a homeomorphism between~$T(x)$ and~$I_{x}^\ast$. 
\end{remark}

The following result is elementary. It will be crucial for us later in this text, by allowing us to consider the equivalence relation on $\Ahyb{n}{k}$ induced by trajectories. 

\begin{lemma}\label{lem:TxTy}
Let $x,y \in \Ahyb{n}{k}$. If $y \in T(x)$, then $T(y)=T(x)$.

In particular, trajectories of points form a partition of~$\Ahyb{n}{k}$.
\qed
\end{lemma}

We may also consider trajectories of arbitrary parts of~$\Ahyb{n}{k}$.

\begin{definition}\label{def:T(V)} 
Let $V$ be a subset of $\Ahyb{n}{k}$. The \emph{trajectory} of~$V$ is the set
\[ T(V) := \bigcup_{x\in V} T(x)  \subseteq \Ahyb{n}{k}.\]
\end{definition}

\begin{remark}\label{rem:TVTV'}
For $V, V' \subseteq \Ahyb{n}{k}$, we have $T(V\cap V') = T(V) \cap T(V')$ and $T(V\cup V') = T(V) \cup T(V')$.
\end{remark}

\subsubsection{Flow-connected and flow-open sets}

We introduce some properties ensuring that the structure sheaf is preserved by the flow.

\begin{definition}
Let $V$ be a subset of $\Ahyb{n}{k}$. We say that $V$ is \emph{flow-connected} if, for each $x\in V$, $T(x) \cap V$ is connected.
\end{definition}

\begin{remark}\label{rem:flowconnected}
Intersections of flow-connected subsets are flow-connected. Indeed, by Remark~\ref{rem:shapeT(x)}, connected sets of the form $T(x) \cap V$ are intervals, hence their connectedness is preserved by intersection.
\end{remark}

\begin{lemma}\label{lem:basisflowconnected}
Flow-connected open subsets form a basis of the topology of $\Ahyb{n}{k}$.
\end{lemma}
\begin{proof}
By definition, a basis of the topology of $\Ahyb{n}{k}$ is given by sets of the form
\[\{x \in \Ahyb{n}{k} : \forall i=1,\dotsc,m, \ r_{i} < \abs{P_{i}(x)} < s_{i} \},\]
for $m \in \N$, $P_{1},\dotsc,P_{m} \in k[T_{1},\dotsc,T_{n}]$ and $r_{1},\dotsc,r_{m},s_{1},\dotsc,s_{m} \in \R$. Those sets are flow-connected.
\end{proof}

\begin{definition}\label{def:flowopen}
Let $V$ be a subset of $\Ahyb{n}{k}$. We say that $V$ is \emph{flow-open} if the map 
\[ 
\begin{array}{ccccc}
\Phi_{V} & \colon & D(\Ahyb{n}{k}) \cap (V \times \R_{>0}) & \too &  T(V)\\
&& (x,\alpha) & \mapstoo & x^\alpha
\end{array}
\]
induced by~$\Phi$ is open.
\end{definition}

\begin{example}\label{ex:flowopen}
\begin{enumerate}[i)]
\item For each $x \in \Ahyb{n}{k}$ and each interval $I \subseteq \R$, $\Phi(x,I \cap I_{x}^\ast)$ is flow-open.
\item By Proposition~\ref{prop:Phicontinuous}, each open subset of $\Ahyb{n}{k}$ is flow-open.
\end{enumerate}
\end{example}

The following properties are direct consequences of the definition.

\begin{lemma}\label{lem:propertiesflowopen}
Let $V$ be a flow-open subset of $\Ahyb{n}{k}$. 
\begin{enumerate}[i)]
\item For each open subset~$V'$ of~$V$, $V'$ is flow-open and $T(V')$ is open in~$T(V)$.
\item For each interval~$I\subseteq \R$, $\Phi(V,I)$ is flow-open. In particular, $T(V)$ is flow-open.
\item For each subset $W$ of $\Ahyb{n}{k}$, $V \cap T(W)$ is flow-open.
\end{enumerate}
\qed
\end{lemma}

\begin{proposition}\label{prop:restrictionflotO}
Let $U$ be a flow-connected subset of $\Ahyb{n}{k}$. Let $V$ be a flow-connected subset of $\Ahyb{n}{k}$ such that $U \subseteq V \subseteq T(U)$.

Let $f \in \cO(V)$. For each $x\in U$ and each $\alpha\in I_{x}^\ast$ such that $x^\alpha\in V$, we have
\[ f(x^\alpha) = \lambda_{x^\alpha,x}(f(x)) \textrm{ and } \abs{f(x^\alpha)} = \abs{f(x)}^\alpha.\]

The restriction map $\cO(V) \to \cO(U)$ is injective.

If, moreover, $U$ is flow-open, then the restriction map $\cO(V) \to \cO(U)$ is an isomorphism.
\end{proposition}
\begin{proof}
We closely follow the proof of \cite[Proposition~1.3.10]{A1Z}, which deals with the case where~$U$ is a singleton. We include details for the convenience of the reader.

Let $x\in U$. Set 
\[ I_{V,x}^\ast := \{ \alpha \in I_{x}^\ast : x^\alpha \in V\}.\]
Since~$V$ is flow-connected, $I_{V,x}^\ast$ is an interval. Set
\[ J :=\{ \alpha \in I_{V,x}^\ast : f(x^\alpha) = \lambda_{x^\alpha,x}(f(x))\}.\]

Let us prove that~$J$ is open. Let $\alpha \in J$. Let $W$ be a neighborhood of~$x^\alpha$ in $\Ahyb{n}{k}$. By continuity of the flow (see Proposition~\ref{prop:Phicontinuous}), there exists an open neighborhood~$L$ of~$\alpha$ in~$I_{x}^\ast$ such that $\{x^\beta : \beta \in L\} \subseteq W$. Up to shrinking~$W$, we may assume that~$f$ is defined on~$W$. By definition of the structure sheaf, up to shrinking~$W$ again, we may assume that there exists a sequence $(R_{i})_{i\in \N}$ of elements of $k(T_{1},\dotsc,T_{n})$ without poles on~$W$ that converges uniformly to~$f$ on~$W$. For each $\beta\in L$ and each $i\in \N$, we have $R_{i}(x^{\beta}) = \lambda_{x^\beta,x^\alpha}(R_{i}(x^\alpha))$. By passing to the limit, we deduce that 
\[f(x^{\beta}) = \lambda_{x^\beta,x^\alpha}(f(x^\alpha)) = \lambda_{x^\beta,x}(f(x)),\]
hence $L\subseteq J$.

Let us prove that~$J$ is closed. Let $\alpha \in I^\ast_{V,x} - J$. Set $m := \abs{f(x^\alpha) - \lambda_{x^\alpha,x}(f(x))} >0$. As before, there exist a neighborhood~$W$ of~$x^\alpha$ (resp. $V$ of $x$) in $\Ahyb{n}{k}$ and a sequence $(R_{i})_{i\in \N}$ (resp. $(S_{i})_{i\in \N}$) of elements of $k(T_{1},\dotsc,T_{n})$ without poles on~$W$ (resp. $V$) that converges uniformly to~$f$ on~$W$ (resp. $V$). 
Let $\eps >0$ such that $\max(\eps, \eps^\alpha) < m/6$.  
There exists $j\in \N$ such that $\norm{f - R_{j}}_{W} < \eps$ and $\norm{f - S_{j}}_{V} < \eps$. For $\beta \in I^\ast_{V,x}$ close enough to~$\alpha$, we have $x^\beta \in W$ and
\begin{equation}\tag{RS}\label{eq:RS} 
f(x^\beta) - \lambda_{x^\beta,x}(f(x))
=   f(x^\beta) - S_{j}(x^\beta) + S_{j}(x^\beta) - R_{j}(x^\beta) + \lambda_{x^\beta,x}(R_{j}(x)) - \lambda_{x^\beta,x}(f(x)).
\end{equation}
Note that $\abs{f(x^\beta) - S_{j}(x^\beta)} \le \eps$ and $\abs{\lambda_{x^\beta,x}(R_{j}(x)) - \lambda_{x^\beta,x}(f(x))} \le \eps^\beta$. 

Applying \eqref{eq:RS} with $\beta=\alpha$ shows that $\abs{S_{j}(x^\alpha) - R_{j}(x^\alpha) } \ge 2m/3$. By continuity of $\abs{S_{j}-R_{j}}$, for $\beta \in I^\ast_{V,x}$ close enough to~$\alpha$, we have $\abs{S_{j}(x^\beta) - R_{j}(x^\beta) } \ge m/2$. Applying \eqref{eq:RS} again for $\beta  \in I^\ast_{V,x}$ close enough to~$\alpha$, we obtain
\[ \abs{f(x^\beta) - \lambda_{x^\beta,x}(f(x))} \ge \frac m 6 >0.\]
It follows that $J$ is closed. 

Since $J$ contains~1, it is non-empty. By connectedness of~$I_{V,x}^\ast$, we have $J = I_{V,x}^\ast$. The equality of absolue values in the first part of the statement follows. It implies that the restriction map $\cO(V) \to \cO(U)$ is injective.

\medbreak

Assume that $U$ is flow open and let us prove that the restriction map $\cO(V) \to \cO(U)$ is surjective. We may assume that $V=T(U)$. Let $g\in \cO(U)$. Let $y\in T(U)$. There exists $x\in U$ such that $y \in T(x)$. Set
\[ f(y) := \lambda_{y,x}(g(x)).\]
By the first part of the statement applied with $V=U$, this does not depend on the choice of~$x$ and we have $f_{\vert U} = g$. It remains to prove that $f \in \cO(T(U))$. 

Let $y\in T(U)$. Let $x\in U$ and $\alpha \in I_{x}^\ast$ such that $y=x^\alpha$. There exists an open neighborhood~$W$ of~$x$ in~$\Ahyb{n}{k}$, an element $g'\in \cO(W)$ with $g'_{\vert W\cap U}=g_{\vert W\cap U}$, and a sequence  $(R_{i})_{i\in \N}$ of elements of $k(T_{1},\dotsc,T_{n})$ without poles on~$W$ that converges uniformly to~$g'$ on~$W$. 

For each $z \in T(W)$, there exists $w\in W$ such that $z \in T(w)$. Set
\[ h'(z) := \lambda_{z,w}(g'(w)).\]
By the first part of the statement applied with $U=W$ and $V=T(W)$, this does not depend on the choice of~$w$. Let $\eps \in (0,\alpha)$. A direct computation shows that the sequence  $(R_{i})_{i\in \N}$ converges uniformly to~$h'$ on
\[ \{ w^\beta : (w,\beta) \in D(\Ahyb{n}{k}) \cap (W \times (\alpha-\eps,\alpha+\eps))\}.\]
By openness of the flow (see Proposition~\ref{prop:Phicontinuous}), the latter is a neighborhood of $x^\alpha=y$ in $\Ahyb{n}{k}$, and it follows that $h'$ belongs to~$\cO(W')$ for some open neighborhood~$W'$ of~$y$ in $\Ahyb{n}{k}$.

Note that $f$ coincides with~$h'$ on $W' \cap T(U)$. Since~$U$ is flow-open, the latter is a neighborhood of~$y$ in~$T(U)$. The result follows.
\end{proof}

\begin{remark}
Without the assumption that $U$ is flow-open, the result of Proposition~\ref{prop:restrictionflotO} may fail. Consider for instance the space $\Ahyb{1}{k}$ with coordinate~$X$ and its subsets 
\[ U_{-}:= \Big\{x \in \Ahyb{1}{k} : \pr(x) < \frac12,\ \abs{X(x)} < 1\Big\}\]
and
\[U_{+} := \{x \in \Ahyb{1}{k} : \pr(x) = 1 ,\ \abs{X(x)} =1\}.\]
The spaces~$U_{-}$ and $U_{+}$ are disjoint, but the trajectory of~$U_{-}$ is 
\[ T(U_{-})= \{x \in \Ahyb{1}{k} : \pr(x) \le 1,\ \abs{X(x)} < 1\},\]
whose closure meets~$U_{+}$. As a result, the morphism $\cO(T(U_{-}\cup U_{+})) \to  \cO(U_{-}\cup U_{+})$ is not surjective.
\end{remark}

\begin{lemma}\label{lem:syFy}
Let $V$ be an open subset of~$\Ahyb{n}{k}$. Let $p\in \N$ and let $\cF$ be a finite type subsheaf of~$\cO_{V}^p$. Let $s\in \cO(V)^p$. For each $x \in \Ahyb{n}{k}$, the set
\[\{ y \in V\cap T(x) : s_{y} \in \cF_{y}\}\]
is open and closed in~$V\cap T(x)$.
\end{lemma}
\begin{proof}
Let $x \in \Ahyb{n}{k}$. Set $T' := \{ y \in V\cap T(x) : s_{y} \in \cF_{y}\}$. By definition, it is open in~$V\cap T(x)$.

Let us prove that $(V\cap T(x)) - T'$ is also open in~$V\cap T(x)$. Let $z \in (V\cap T(x)) - T'$. Then $s_{z} \notin \cF_{z}$. Since~$\cF$ is of finite type, there exist an open neighborhood~$W$ of~$z$ in~$V$ and sections $s_{1},\dotsc,s_{r}$ of~$\cF(W)$ that generate~$\cF$ on~$W$. By Lemma~\ref{lem:basisflowconnected}, we may assume that~$W$ is flow-connected. Assume, by contradiction, that $W \cap T' \ne \emptyset$ and let $z' \in W \cap T'$. Then, we have $s_{z'} \in \cF_{z'}$, hence there exist an open neighborhood~$W'$ of~$z'$ in~$W$ and $a_{1},\dotsc,a_{r} \in \cO(W')$ such that $s = \sum_{i=1}^r a_{i} s_{i}$ on~$W'$. We may assume that~$W'$ is flow-connected. 
By surjectivity of the restriction map $\cO(T(W')) \to \cO(W')$ (see Proposition~\ref{prop:restrictionflotO}), the $a_{i}$'s extend to elements of~$\cO(T(W'))$, which we denote identically. By injectivity of the restriction map $\cO(W\cap T(W')) \to \cO(W')$ (see Proposition~\ref{prop:restrictionflotO}), the equality $s = \sum_{i=1}^r a_{i} s_{i}$ still holds on $W \cap T(W')$, hence in the neighborhood of~$z$. It follows that $s_{z} \in \cF_{z}$, which is a contradiction. This concludes the proof.
\end{proof}

\begin{lemma}\label{lem:sy=0F}
Let $V$ be an open subset of~$\Ahyb{n}{k}$ and let~$\cF$ be a coherent sheaf on~$V$. Let $s\in \cF(V)$. For each $x \in \Ahyb{n}{k}$, the set
\[\{ y \in V\cap T(x) : s_{y} = 0\}\]
is open and closed in~$V\cap T(x)$.
\end{lemma}
\begin{proof}
Let $x \in \Ahyb{n}{k}$. Set $T' := \{ y \in V\cap T(x) : s_{y} = 0\}$. By definition, it is open in~$V\cap T(x)$.

Let us prove that $(V\cap T(x)) - T'$ is also open in~$V\cap T(x)$. Let $z \in (V\cap T(x)) - T'$. Then $s_{z} \ne 0$. Since~$\cF$ is coherent, there exists an open neighborhood~$W$ of~$z$ in~$V$ such that $\cF \simeq \cO_{W}^p/\cM$, where $\cM$ is a finite type subsheaf of~$\cO_{W}^p$. Up to shrinking~$W$, we may assume that $s_{\vert W}$ lifts to an element~$s'$ of~$\cO^p(W)$. Since $s_{z} \ne 0$, we have $s'_{z} \notin \cM_{z}$. By Lemma~\ref{lem:syFy}, the last property holds in a neighborhood of~$z$ in~$V\cap T(x)$. The result follows.
\end{proof}

\begin{theorem}\label{th:FVFUaffine}
Let $U$ and $V$ be flow-connected flow-open subsets of $\Ahyb{n}{k}$ such that $U \subseteq V \subseteq T(U)$ and $U$ is open in~$V$. 
For each coherent sheaf~$\cF$ on~$V$, the restriction map $\cF(V) \to \cF(U)$ is an isomorphism. 

Moreover, the restriction map induces an equivalence between the categories of coherent sheaves on~$V$ and~$U$.
\end{theorem}
\begin{proof}
Let $\cF$ be a coherent sheaf on~$V$. By Lemma~\ref{lem:sy=0F}, the restriction map $\cF(V) \to \cF(U)$ is injective. Let us prove that it is surjective. Let $s\in \cF(U)$. By Zorn's lemma, there exists a maximal flow-connected open subset~$U'$ of~$V$ containing~$U$ to which~$s$ extends. Up to replacing~$U$ by~$U'$, we may assume that~$s$ does not extend beyond~$U$.

Assume, in view of a contradiction that $U\ne V$. Let $x$ be a point in the boundary of~$U$ in~$V$. There exist a open neighborhood~$V_{x}$ of~$x$ in~$V$ and $t_{1},\dotsc,t_{p} \in \cF(V_{x})$ that generate~$\cF$ on~$V_{x}$. We may assume that~$V_{x}$ is flow-connected. Let $y \in V_{x} \cap T(x) \cap U$. There exist an open neighborhood~$U_{y}$ of~$y$ in~$V_{x}\cap U$ and $a_{1},\dotsc,a_{p} \in \cO(U_{y})$ such that $s_{\vert U_{y}} = \sum_{i=1}^p a_{i} t_{i,\vert U_{y}}$. We may assume that~$U_{y}$ is flow-connected. Note that $U_{y}$ is also flow-open, by Lemma~\ref{lem:propertiesflowopen}, since it is open in~$U$. By Proposition~\ref{prop:restrictionflotO}, the $a_{i}$'s extend to elements of~$T(U_{y})$, which we denote in the same way. Set $W_{x} := V_{x} \cap T(U_{y})$ and $s' := \sum_{i=1}^p a_{i,\vert W_{x}} t_{i,\vert W_{x}}$. By Lemma~\ref{lem:propertiesflowopen}, the set~$W_{x}$ is a flow-connected open neighborhood of~$x$ in~$V$ contained in~$T(U_{y})$ and the element~$s'$ belongs to~$\cF(W_{x})$. By Lemma~\ref{lem:sy=0F}, $s'$ coincides with~$s$ on $W_{x}\cap U$. As a result, we can extend~$x$ to $U\cup W_{x}$, which provides the desired contraction. We have just proved that $U = V$, which implies the claimed surjectivity. 

\medbreak

Let us now consider the functor 
\[ R_{U,V} \colon \cF \mapstoo \cF_{\vert U}\] 
from the category of coherent sheaves on~$V$ to that of coherent sheaves on~$U$. 

Let $\cF,\cG$ be coherent sheaves on~$V$. The sheaf $\sHom(\cF,\cG)$ is coherent, hence, by the first part of the statement, the restriction from~$V$ to~$U$ induces an isomorphism
\[ \Hom(\cF,\cG) = \sHom(\cF,\cG)(V) \simto  \sHom(\cF,\cG)(U) = \Hom(\cF_{\vert U},\cG_{\vert U}).\]
In other words, the functor~$R_{U,V}$ is fully faithful.

It remains to prove that~$R_{U,V}$ is essentially surjective. Let~$\cF$ be a coherent sheaf on~$U$. As in the first part of the proof, using Zorn's lemma, we may assume that~$\cF$ admits no extensions as a coherent sheaf to a bigger flow-connected open subset of~$V$. 

Assume, in view of a contradiction, that $U \ne V$. Let $x$ be a point in the boundary of~$U$ in~$V$. Since $U \subseteq V \subseteq T(U)$, there exists $y\in U$ such that $y \in T(x)$. By assumption, $\cF$ is coherent, hence there exists an open neighborhood~$W$ of~$y$ in~$U$ such that $\cF_{\vert W}$ is isomorphic to a sheaf of the form $\cO_{W}^{p}/\cM$, where $\cM$ is a finite type submodule of $\cO_{W}^{p}$. Up to shrinking~$W$, we may assume that it is flow-connected and that~$\cM$ is generated by $s_{1},\dotsc,s_{r} \in \cO(W)^p$. Set $W' := T(W) \cap V$. By Lemma~\ref{lem:propertiesflowopen}, it is an open neighborhood of~$x$ in~$V$. By Proposition~\ref{prop:restrictionflotO}, $s_{1},\dotsc,s_{r}$ extend to elements $s'_{1},\dotsc,s'_{r}$ in $\cO(W')^p$. Denote by~$\cM'$ the submodule of $\cO_{W'}^{p}$ that they generate and set $\cF' := \cO_{W'}^{p}/\cM'$. 

By construction, $\cF'$ is a coherent sheaf on~$W'$ that it is isomorphic to~$\cF$ on~$W$. By full-faithfulness of the restriction functor, the isomorphism extends to $T(W)\cap U = W' \cap U$. As a result, we may glue~$\cF$ and~$\cF'$ to a coherent sheaf on $U \cup W'$, which provides a contradiction. It follows that $U=V$, which concludes the proof of essential surjectivity.
\end{proof}

\subsubsection{Flowing spaces}

We explain how to generalize the flow and the related notions to more general spaces.

\begin{definition}
Let $V$ be a subset of $\Ahyb{n}{k}$. We say that $V$ is \emph{flowing} if, for each $x\in V$ and $\alpha\in I_{x}^\ast$, we have $x^\alpha \in V$.
\end{definition}

\begin{remark}\label{rem:flowing}
Flowing subsets are flow-connected. Intersections and unions of flowing subsets are flowing.
\end{remark}

\begin{definition}
A \emph{flowing local $k^\hyb$-analytic model} is the data of an integer~$n$, a flowing open subset~$U$ of~$\Ahyb{n}{k}$ and a closed analytic subset of~$U$.

A \emph{flowing $k^\hyb$-analytic space} is a $k^\hyb$-analytic space obtained by gluing flowing local $k^\hyb$-analytic models along flowing open subsets.
\end{definition}

\begin{example}
Let $X$ be a variety over~$k$. By construction, its analytification~$X^\hyb$ over~$k^\hyb$ is a flowing $k^\hyb$-analytic space.
\end{example}

The definitions given at the beginning of Section~\ref{sec:flow}: $x^\alpha$, $\lambda_{x^\alpha,x}$, the flow $\Phi$, the trajectory $T(x)$, flow-connected and flow-open subsets, etc. immediately extend to arbitrary flowing $k^\hyb$-analytic spaces. We will use the same notation and terminology in this more general setting. All the properties stated above generalize easily. We only record the last one here, for later reference.

\begin{theorem}\label{th:coherentflow}
Let $X$ be a flowing $k^\hyb$-analytic space. Let $U$ and $V$ be flow-connected flow-open subsets of $X$ such that $U \subseteq V \subseteq T(U)$ and $U$ is open in~$V$. For each coherent sheaf~$\cF$ on~$V$, the restriction map $\cF(V) \to \cF(U)$ is an isomorphism. 

Moreover, the restriction map induces an equivalence between the categories of coherent sheaves on~$V$ and~$U$.
\qed
\end{theorem}

Let us extend this result to other situations. 

\begin{definition}
Let $X$ be a flowing $k^\hyb$-analytic space. A subset~$V$ of~$X$ is said to be \emph{over-flow-connected} if it admits a basis of open neighborhoods in $T(V)$ made of flow-connected flow-open subsets.
\end{definition}

\begin{remark}
Over-flow-connected subsets are flow-connected.
\end{remark}

\begin{example}
For each $x \in X$ and each interval $I \subseteq \R$, $\Phi(x,I \cap I_{x}^\ast)$ is over-flow-connected (see Example~\ref{ex:flowopen}).
\end{example}

Let us now exhibit another family of examples. 

\begin{definition}
Let $B$ be an affine variety over~$k$. A compact subset~$V$ of~$B^\hyb$ is called a \emph{rational domain} if there exist $f_{1},\dotsc,f_{m},g \in \cO(B^\hyb)$, with $(f_{1},\dotsc,f_{m},g)=(1)$, and real numbers $r_{1},\dotsc,r_{m} \in \R_{>0}$ such that
\[ V = \{ x \in B^\hyb : \forall i = 1,\dotsc,m,\ \abs{f_{i}(x)} \le r_{i} \, \abs{g(x)}\}.\]
\end{definition}

\begin{lemma}\label{lem:basisrational}
Let $B$ be an affine variety over~$k$. 

\begin{enumerate}[i)]
\item A finite intersection of rational domains of~$B^\hyb$ is a rational domain.
\item Each rational domain of~$B^\hyb$ admits a basis of neighborhoods in~$B^\hyb$ made of flow-connected open subsets. In particular, each rational domain of~$B^\hyb$ is over-flow-connected.
\end{enumerate}
\end{lemma}
\begin{proof}
i) The arguments of the classical proof (see \cite[Proposition 7.2.3/7]{BGR}) still applies here.

\medbreak

ii) Let $V$ be a rational domain of~$B^\hyb$. By definition, there exist $f_{1},\dotsc,f_{m},g \in \cO(B^\hyb)$, with $(f_{1},\dotsc,f_{n},g)=(1)$, and real numbers $r_{1},\dotsc,r_{m} \in \R_{>0}$ such that
\[ V = \{ x \in B^\hyb : \forall i = 1,\dotsc,m,\ \abs{f_{i}(x)} \le r_{i} \, \abs{g(x)}\}.\]

Let us identify $B$ to a Zariski-closed subset of an affine space~$\A^N_{k}$, hence $B^\hyb$ to a closed analytic subset of~$\Ahyb{N}{k}$. Since $V$ is compact, there exists an open disc~$D_{<}$ containing~$V$ in~$\Ahyb{N}{k}$. Denote by~$D_{\le}$ the closed disc with same center and radius as~$D_{<}$. Note that $D_{\ge} \cap B^\hyb$ is a rational domain of~$B^\hyb$.

The function~$g$ does not vanish on~$V$. Set 
\[ s := \inf(\{\abs{g(y)}  : y \in V\}).\]
Since~$V$ is compact, this infimum is attained, hence $s>0$. Let $t \in (0,s)$. The sets
\[ E_{\ge} := \{ x \in B^\hyb : \abs{g(x)} \ge t\} \textrm{ and } E_{>} := \{ x \in B^\hyb : \abs{g(x)} > t\}\]
contain~$V$.

Let $U$ be a neighborhood of~$V$ in~$B^\hyb$. We need to prove that there exists a contains a flow-connected open subset~$W$ of~$B^\hyb$ such that $V \subseteq W \subseteq U$. We may assume that $U \subseteq D_{\le} \cap E_{\ge}$. The function~$g$ does not vanish on~$D_{\le} \cap E_{\ge}$. Consider the map
\[ F \colon y \in D_{\le} \cap E_{\ge} \mapstoo \max_{1\le i \le m} \Big( \frac{\abs{f_{i}(y)}}{r_{i}\, \abs{g(y)}}\Big) \in \R_{\ge 0}.\]  
Set \[a :=  \inf( \{\abs{F(z)}  : z \in (D_{\le} \cap E_{\ge}) \setminus U\} ).\]
Since $(D_{\le} \cap E_{\ge}) \setminus U$ is compact and non-empty, this infimum is attained, hence $a>1$. It follows that the set
\[ \{ y \in D_{\le} \cap E_{\ge} : \forall i = 1,\dotsc,m,\ \abs{f_{i}(y)} < a\, r_{i} \, \abs{g(y)}\}\]
is contained in~$U$, hence the set
\[W := \{ y \in D_{<} \cap E_{>} : \forall i = 1,\dotsc,m,\ \abs{f_{i}(y)} < a\, r_{i} \, \abs{g (y)}\}\]
satisfies all the conditions.

\medbreak

The last part of the statement follows from Example~\ref{ex:flowopen}, ii) and Lemma~\ref{lem:propertiesflowopen}, iii).
\end{proof}

\begin{corollary}\label{cor:coherentflowcompact}
Let $X$ be a flowing $k^\hyb$-analytic space. Let $U$ and $V$ be compact over-flow-connected subsets of~$X$ such that $U \subseteq V \subseteq T(U)$. For each coherent sheaf~$\cF$ on~$V$, the restriction map $\cF(V) \to \cF(U)$ is an isomorphism. 

Moreover, the restriction map induces an equivalence between the categories of coherent sheaves on~$V$ and~$U$.
\end{corollary}
\begin{proof}
Let $\cF$ be a coherent sheaf on~$V$. By Proposition~\ref{prop:extensioncoherentsheafcompact}, it extends to a neighborhood~$W$ of~$V$ in~$T(V)$. Consider a flow-connected flow-open open neighborhood~$V'$ of~$V$ in~$W$ and a flow-connected flow-open open neighborhood~$U'$ of~$U$ in~$V'$. By Theorem~\ref{th:coherentflow}, the restriction map $\cF(V') \to \cF(U')$ is an isomorphism. By assumption, such~$U'$'s and~$V'$'s form bases of neighborhoods of~$U$ and~$V$. It follows that the restriction map $\cF(V) \to \cF(U)$ is an isomorphism. 

The last part of the result is proved similarly. 
\end{proof}

\begin{corollary}\label{cor:coherentTx}
Let $X$ be a flowing $k^\hyb$-analytic space. Let $x \in X$. For each coherent sheaf~$\cF$ on~$T(x)$, the restriction map $\cF(T(x)) \to \cF_{x}$ is an isomorphism. 

Moreover, the restriction map induces an equivalence between the category of coherent sheaves on~$T(x)$ and the category of $\cO_{x}$-modules of finite type.
\end{corollary}
\begin{proof}
Let $I$ be a compact subinterval of $I_{x}^\ast$ containing~1. By Lemma~\ref{lem:propertiesflowopen}, ii), $\{x\}$ and $\Phi(\{x\},I)$ are over-flow-connected. It then follows from Corollary~\ref{cor:coherentflowcompact} that the result holds with $T(x)$ replaced by $\Phi(\{x\},I)$.

Since $T(x)$ may be exhausted by subsets of the form $\Phi(\{x\},I)$, we deduce that the result holds as stated.
\end{proof}

\subsection{Analytification}

We investigate the analytification functor in the hybrid setting.

\begin{notation}\label{nota:Xhyb}
Let $X$ be a scheme locally of finite type over~$k$. 

We denote by~$X^\hyb$ the analytification of~$X$ over~$k_{\hyb}$.

For $\eps\in(0,1]$, we denote by $X_{[0,\eps]}^\an$ the analytification of~$X$ over the Banach ring $(k,\max(\va_{0},\va_{\eps}))$. In particular, $X_{[0,1]}^\an = X^\hyb$.

We denote by~$X^\an_{0}$ the analytification of~$X$ over~$k_{0}$. With Notation~\ref{nota:Xeps}, we have a canonical homeomorphism $X^\hyb_{0} = X^\an_{0}$. We denote by~$X^\hyb_{0^\dag}$ the space~$X^\hyb_{0}$ endowed with the overconvergent structure sheaf inherited from~$X^\hyb$. It is not isomorphic to the locally ringed space $X^\an_{0}$ in general.

For $\eps\in(0,1]$, we denote by~$X^\an_{\eps}$ the analytification of $X\otimes_{k} \hat k_{\eps}$ over $k_{\eps}$. With Notation~\ref{nota:Xeps}, we have a canonical homeomorphism $X^\hyb_{\eps} = X^\an_{\eps}$. We endow $X^\hyb_{\eps}$ with the overconvergent structure sheaf inherited from $X^\hyb$.
\end{notation}

\begin{lemma}\label{lem:hybaneps}
Let $X$ be a scheme locally of finite type over~$k$. Let $\eps\in(0,1]$. The natural morphism of locally ringed spaces $\iota \colon X^\an_{\eps} \to X^\hyb_{\eps}$ is an isomorphism.

Moreover, for each open subset $U$ of $X^\hyb_{\eps}$ and each coherent sheaf~$\cF$ on~$U$, we have $\iota^\ast\cF(\iota^{-1}(U)) = \cF(U)$.
\end{lemma}
\begin{proof}
When $X$ is an affine space, the first part of the statement is \cite[Proposition 3.4.6]{A1Z}. The general case follows. The proof is very similar to that of Proposition~\ref{prop:restrictionflotO}.

The second part of the statement follows by working locally and using a finite presentation of~$\cF$.
\end{proof}

\begin{lemma}\label{lem:basisflowsaturated}
Let $X$ be a scheme locally of finite type over~$k$. Let $\eps\in(0,1]$. Each point~$x$ of~$X^\hyb_{\eps}$ has a basis of neighborhoods~$\Vc_{x}$ in~$X^\hyb_{\eps}$ such that each $V\in \Vc_{x}$ satisfies the following properties:
\begin{enumerate}[i)]
\item $V$ is compact and over-flow-connected;
\item $T(V)$ may be written as an increasing union of compact over-flow-connected subsets containing~$V$.
\end{enumerate}
\end{lemma}
\begin{proof}
Let $x \in X^\hyb_{\eps}$. The question being local, we may assume that $X$ is affine. Let us consider a closed embedding $j \colon X \to \A^N_{k}$ and fix coordinates $T=(T_{1},\dotsc,T_{N})$ on~$\A^N_{k}$.

By definition of the topology, $j^\hyb(x)$ admits a basis of neighborhoods in $(\Ahyb{N}{k})_{\eps}$  of the form 
\[ W = \bigcap_{i=1}^p \big\{ y \in (\Ahyb{N}{k})_{\eps} : \abs{P_{i}(y)} \le s_{i}\big\} \cap \bigcap_{i=p+1}^q \big\{ y \in (\Ahyb{N}{k})_{\eps} : r_{i} \le \abs{P_{i}(y)} \le s_{i}\big\},\]
with $P_{1},\dotsc,P_{q} \in k[T]$, $r_{p+1},\dotsc,r_{q},s_{1},\dotsc,s_{q} \in \R_{>0}$, $\abs{P_{i}(j^\hyb(x))} < s_{i}$ for each $i \in \{1,\dotsc,p\}$ and $r_{i} < \abs{P_{i}(j^\hyb(x))} < s_{i}$ for each $i \in \{p+1,\dotsc,q\}$. We may assume that the coordinate functions $T_{1},\dotsc,T_{N}$ belong to the family $(P_{i})_{1\le i\le q}$. Let us now prove that the set $V := j^{-1}(W)$ satisfies properties~i) and~ii) of the statement.

\medbreak

i) Since the family $(P_{i})_{1\le i\le q}$ contains the coordinate functions $T_{1},\dotsc,T_{N}$, $W$ is compact, and so is $V := j^{-1}(W)$. By Lemma~\ref{lem:basisrational}, $W$ and $V$ are over-flow-connected too.

\medbreak

ii) Since $k$ is not trivially valued, there exists $\alpha \in k$ such that $0 < \abs{\alpha} < 1$. Up to multiplying the $P_{i}$'s by some power of~$\alpha$, and the corresponding~$r_{i}$'s (when they exist) and~$s_{i}$'s by the same power of~$\abs{\alpha}$, we may assume that $r_{p+1},\dotsc,r_{p},s_{1},\dotsc,s_{q} \in (0,1)$. In this case, for each $i\in \{p+1,\dotsc,q\}$ (resp. $i \in \{1,\dotsc,q\}$), there exists $u_{i} \in \R_{>0}$ (resp. $v_{i} \in \R_{>0}$) such that $r_{i} = \abs{\alpha}^{\eps u_{i}}$ (resp. $s_{i} = \abs{\alpha}^{\eps v_{i}}$), so that we have
\[ W = \bigcap_{i=1}^p \big\{ y \in (\Ahyb{N}{k})_{\eps} : \abs{P_{i}(y)} \le \abs{\alpha(y)}^{v_{i}}\big\} \cap \bigcap_{i=p+1}^q \big\{ y \in (\Ahyb{N}{k})_{\eps} : \abs{\alpha(y)}^{u_{i}} \le \abs{P_{i}(y)} \le \abs{\alpha(y)}^{v_{i}}\big\}.\]
We have 
\[ T(W) = \bigcap_{i=1}^p \big\{ y \in (\Ahyb{N}{k})_{>0} : \abs{P_{i}(y)} \le \abs{\alpha(y)}^{v_{i}}\big\} \cap \bigcap_{i=p+1}^q \big\{ y \in (\Ahyb{N}{k})_{>0} : \abs{\alpha(y)}^{u_{i}} \le \abs{P_{i}(y)} \le \abs{\alpha(y)}^{v_{i}}\big\}.\]

For $n\in \N$, set 
\[
T_{n}(W) := T(W) \cap \pr^{-1}([2^{-n},1])
= \{ y \in T(W) : \abs{\alpha(y)}  \ge \abs{\alpha}^{2^{-n}}  \}.
\]
By Lemma~\ref{lem:basisrational}, it is a compact over-flow-connected subset of $\Ahyb{N}{k}$. Let $n_{0}\in \N$ such that $\eps \ge 2^{-n_{0}}$. For $n\ge n_{0}$, $T_{n}(W)$ contains~$W$.

It is now easy to check that the properties of the statement are satisfied by the sequence $(V_{n} := j^{-1}(W_{n}))_{n\ge n_{0}}$.
\end{proof}

\begin{lemma}\label{lem:>0eps}
Let $X$ be a scheme locally of finite type over~$k$. Let $\eps\in(0,1]$. Let $U$ be an open subset of~$X^\hyb_{>0}$ such that $T(U)=U$. For each coherent sheaf~$\cF$ on~$U$, the restriction map $\cF(U) \to \cF(U \cap X^\hyb_{\eps})$ is an isomorphism. 

Moreover, the restriction map induces an equivalence between the categories of coherent sheaves on~$U$ and~$U \cap X^\hyb_{\eps}$.
\end{lemma}
\begin{proof}
Let $x \in U \cap X^\hyb_{\eps}$. By Lemma~\ref{lem:basisflowsaturated}, there exists a compact over-flow-connected neighborhood~$V$ of~$x$ in $U$ such that $T(V)$ may be written a the union of a sequence  $(V_{n})_{n\in \N}$ of compact over-flow-connected subsets containing~$V$. By Corollary~\ref{cor:coherentflowcompact}, the result holds with $U$ and $U \cap X^\hyb_{\eps}$ replaced by $V_{n}$ and $V$ respectively, hence also by $T(V)$ and $V$ respectively.

The subsets~$V$ as above cover $U \cap X^\hyb_{\eps}$, and the subsets $T(V)$ cover $T(U \cap X^\hyb_{\eps}) = T(U) = U$. The result follows.   
\end{proof}

Remarkably, a GAGA theorem holds for hybrid analytifications, without requiring properness assumptions. This was already observed by V.~Berkovich in the trivially valued case, that is to say for the functor $X \mapsto X^\an_{0}$, see \cite[Theorem~3.5.1]{rouge}. We follow his proof.

\begin{theorem}\label{th:GAGAhybrid}
Let $X$ be a scheme locally of finite type over~$k$. 

The functor $\cF \mapsto \cF^\hyb$ realises an equivalence between the categories of coherent sheaves on~$X$ and~$X^\hyb$.

Moreover, for each coherent sheaf~$\cF$ on~$X$ and each $q\in \N$, we have a canonical isomorphism
\[ H^q(X,\cF) \simto H^q(X^\hyb,\cF^\hyb).\]

The same results holds replacing~$X^\hyb$ by~$X_{0^\dag}^\hyb$ or $X_{[0,\eps]}^\an$ for $\eps\in (0,1]$.
\end{theorem}
\begin{proof}
Let us begin with the first part of the statement. We may assume that $X$ is affine, say $X = \Spec(A)$, with $A = k[T_{1},\dotsc,T_{n}]/I$ for a finitely generated ideal~$I$ of~$k[T_{1},\dotsc,T_{n}]$.

By \cite[Lemme~6.5.6]{CTCZ}, the functor $\cF \mapsto \cF^\hyb$ sends coherent sheaves to coherent sheaves. Moreover, by \cite[Th\'eor\`eme~6.6.5]{CTCZ}, the canonical morphism $X^\hyb\to X$ is flat, hence the preceding functor is faithful. 

It follows from \cite[Corollaire~2.8]{EtudeLocale} that we have a canonical identification $k[T_{1},\dotsc,T_{n}] = \cO(\Ahyb{n}{k})$. By \cite[Corollaire~8.5.15]{CTCZ} (applied with infinite radii), the space $\cO(\Ahyb{n}{k})$ has no higher coherent cohomology. By a classical argument (see the proof of \cite[Th\'eor\`eme~8.3.8]{CTCZ}, for instance), we deduce that $A = \cO(X^\hyb)$. Let $\cF$ be a coherent sheaf on~$X$. Writing it as a cokernel of free sheaves of finite rank and using right-exactness of the global section functor, we deduce that $\cF^\hyb(X^\hyb) = \cF(X)$. Applying this result to sheaves of homomorphisms, it follows that the functor $\cF \mapsto \cF^\hyb$ is full.

It remains to prove essential surjectivity. Let $r \in \R_{\ge 1}$ and consider the relative closed disc 
\[D_{\hyb}(r) := \{ x \in \Ahyb{n}{k} : \forall i =1,\dotsc,n,\ \abs{T_{i}(x)}\le r\}.\]
Denote by~$X^\hyb(r)$ its closed analytic subset defined by the ideal~$I$. It identifies to a subset of~$X^\hyb$, and the family $(X^\hyb(r))_{r\ge 1}$ exhausts~$X^\hyb$. By the same argument as above (using \cite[Corollaire~8.2.16]{CTCZ} instead of \cite[Corollaire~8.5.15]{CTCZ}), we have $A = \cO(X^\hyb(r))$. Moreover, by \cite[Corollaire~8.3.10]{CTCZ}, we have an equivalence between the category of coherent sheaves over~$X^\hyb(r)$ and the category of finitely generated modules over $\cO(X^\hyb(r)) = A$. The result follows. 

\medbreak

Let us now prove the second part of the statement. It concerns an arbitrary scheme~$X$ that is locally of finite type over~$k$. Let~$\cF$ be a coherent sheaf on~$X$. 

\smallbreak

\noindent \textit{Step~1:} $X$ is affine

During the proof of the first part of the statement, we have seen that $\cF^\hyb(X^\hyb) = \cF(X)$. Moreover, by \cite[Corollaire~8.5.24]{CTCZ}, for each $q\ge 1$, we have 
\[H^q(X^\hyb,\cF^\hyb) = H^q(X,\cF) =0.\]

\smallbreak

\noindent \textit{Step~2:} $X$ is separated

Cover~$X$ by a family of open affine subschemes $(U_{i})_{i\in I}$. Since $X$ is separated, the multiple intersections of the~$U_{i}$'s are affine too. The result now follows from the result in the affine case, together with the \v Cech-to-cohomology spectral sequence.

\smallbreak

\noindent \textit{Step~3:} $X$ is arbitrary

Cover~$X$ by a family of open affine subschemes $(U_{i})_{i\in I}$. The $U_{i}$'s and their multiple intersections are separated, hence the result of Step~2 applies to them. The result for~$X$ follows again from the \v Cech-to-cohomology spectral sequence.

\medbreak

The results for $X_{0^\dag}^\hyb$ or $X_{[0,\eps]}^\an$ are proven using the same arguments.
\end{proof}

\section{Varieties and analytic spaces over trivially valued fields}\label{sec:triviallyvalued}

In the recent literature, several examples occur where Berkovich analytic spaces over trivially valued fields are used to understand the behaviour at the boundary of algebraic varieties.

\subsection{Generic fibers}

An interesting feature of non-Archimedean analytic spaces is that they can be used to define a notion of generic fiber for formal schemes (with suitable finiteness assumptions) over valuation rings. This idea was first developed by M.~Raynaud (see~\cite{tableronde}), for formal schemes locally topologically of finite presentation, and later extended by P.~Berthelot to a wider class (see~\cite{BerthelotCohomologieRigide}). We point the reader to \cite{BerkovichVC} and \cite{BerkovichVCII} for the corresponding constructions in the setting of Berkovich spaces. 

In this text, we only need generic fibers of formal schemes of a rather specific sort, and we only recall the construction in this setting. We follow the presentation given by A.~Thuillier in~\cite{beth} and borrow his notation.\footnote{In particular, we denote the usual generic fiber of a formal scheme~$\Xk$ by~$\Xk^\beth$ instead of the customary $\Xk_{\eta}$, which is reserved for a smaller space.} 

Remark that the valued field~$k_{0}$ is a valuation ring. We consider formal schemes of over~$k_{0}$ that are \emph{locally algebraic}, \emph{i.e.} locally isomorphic to the completion of a variety along a closed subscheme. Note that they are special formal schemes in the sense of~\cite{BerkovichVCII}. To a locally algebraic formal scheme~$\Xk$, we associate a generic fiber~$\Xk^\beth$ (which is an analytic space over~$k_{0}$) and a map $r_{\Xk} \colon \Xk^\beth \to \Xk_{s}$ which is anticontinuous (the preimage of an open subset is closed). The construction involves three steps. 

\smallbreak

\noindent \textit{Step~1:} $\Xk = X$ is an affine variety

Write $X = \Spec(A)$. Set 
\[X^\beth := \cM(A),\] 
where $A$ is endowed with the trivial norm. The map $r_{X} \colon \cM(A) \to \Spec(A)$ is the reduction map defined as follows: for $x \in \cM(A)$, 
\[ r_{X}(x) := \{a \in A : \abs{a(x)} <1\}.\]

\smallbreak

\noindent \textit{Step~2:} $\Xk$ is the completion of an affine variety~$X$ along a closed subscheme~$Y$ 

Write $X = \Spec(A)$ and $Y=V(I)$, where $I$ is an ideal of~$A$. Set
\[ \Xk^\beth := r_{X}^{-1}(Y) = \{x\in X^\beth : \forall b \in I, \abs{b(x)} <1\}.\]
It is an open subset of~$X^\beth$, from which it inherits an analytic structure. The map $r_{\Xk} \colon r_{X}^{-1}(Y) \to Y$ is that induced by~$r_{X}$.

\smallbreak

\noindent \textit{Step~3:}  $\Xk$ is locally algebraic

In this case, $\Xk^\beth$ and $r_{\Xk}$ are obtained from the construction described in Step~2 by a gluing process. 

\medbreak

Let us mention that A.~Thuillier gives a more precise presentation: he actually states a functorial definition of~$\Xk^\beth$, where $r_{\Xk}$ is enhanced to a morphism $\Xk^\beth \to \Xk$. We refer to \cite[Proposition et D\'efinition 1.3]{beth} for details.

\medbreak

We may now state A.~Thuillier's refined version of the generic fiber.

\begin{definition}
Let~$\Xk$ be a locally algebraic formal scheme~$\Xk$ over~$k_{0}$. Its \emph{generic fiber in the sense of A.~Thuillier} is the analytic space over~$k_{0}$ defined by 
\[\Xk_{\eta} := \Xk^\beth - \Xk_{s}^\beth.\]
\end{definition}

\begin{example}
Let $\Xk$ be the the completion of a variety $X = \Spec(A)$ along a closed subscheme $Y=V(I)$. Then, we have
\[ \Xk_{\eta} =  \{x\in \cM(A) : \forall b \in I, \abs{b(x)} <1\} -  \{x\in \cM(A) : \forall b \in I, b(x) = 0\}.\]
In particular, if $I = (b_{1},\dotsc,b_{r})$, then
\[ \Xk_{\eta} =  \{x\in \cM(A) : 0 < \max_{1\le i\le r}(\abs{b_{i}(x)}) <1\}.\]
\end{example}

\subsection{Generic fibers and analytifications}\label{sec:genericfibers}

Let $X$ be a variety over~$k$. We have just seen a way to obtain an analytic space over~$k_{0}$ from~$X$, by passing to the associated formal scheme and taking the generic fiber~$X^\beth$. Another possibility is to consider the analytification~$X^\an_{0}$ of~$X$ over~$k_{0}$. It is related to the former. Indeed, by \cite[Proposition~1.10]{beth}, $X^\beth$ may be identified to a compact analytic domain of~$X^\an_0$.

To be more precise, let us extend to multiplicative seminorms the classical notion of center of a valuation. Recall that there exists a canonical morphism $\rho \colon X_{0}^\an \to X$ (see Section~\ref{sec:analytification}). Let $x \in X^\an_0$. By Remark~\ref{rem:pairs}, the point~$x$ may be described by the data of the point~$\rho(x) \in X$ and an absolute value~$\va_{x}$ on the residue field~$\kappa(\rho(x))$. We will, somewhat abusively, denote by~$\kappa(\rho(x))^\circ$ the subring of elements of~$\kappa(\rho(x))$ that are bounded by~1 for~$\va_{x}$.

\begin{definition}\label{def:center}
Let $x\in X_{0}^\an$. We say that the seminorm associated to~$x$ (or the point~$x$) \emph{has a center on~$X$} if there exists a morphism $\lambda^\circ_{x} \colon \Spec (\kappa(\rho(x)))^\circ \to X$ that makes the following diagram 
\[
\begin{tikzcd}
\Spec (\kappa(\rho(x))) \ar[r, "\lambda_{x}"] \ar[d] & X \\
\Spec (\kappa(\rho(x))^\circ) \ar[ru, dashed, "\lambda_{x}^\circ"'] &
\end{tikzcd}
\]
commutative.

Since $X$ is separated, as soon as it exists, the morphism $\lambda^\circ_{x}$ is unique. In this case, the image of the closed point of $\Spec (\kappa(\rho(x))^\circ)$ is called the \emph{center} of the seminorm associated to~$x$ (or the point~$x$). 

For each subset $U$ of~$X$, we say that the seminorm associated to~$x$ (or the point~$x$) \emph{has a center on~$U$} if it has a center on~$X$, and its center belongs to~$U$.
\end{definition}

The following result is a direct consequence of the definitions.

\begin{lemma}\label{lem:center}
Let $x\in X_{0}^\an$. 

\begin{enumerate}[i)]
\item The point~$x$ belongs to~$X^\beth$ if, and only if, it has a center on~$X$.

\item For each subset~$U$ of~$X$, the point~$x$ belongs to~$r_{X}^{-1}(U)$ if, and only if, it has a center on~$U$.
\end{enumerate}
\qed
\end{lemma}

It follows from the valuative criterion of properness that $X^\beth = X^\an_0$ if, and only if, $X$ is proper (see \cite[Proposition~1.10]{beth}). 

\medbreak

Following V.~Berkovich's suggestion in his correspondence with V.~Drinfeld (see \cite{BerkovichtoDrinfeld} or \cite[Section~3.3.2]{BBT}), we introduce a notation for the set of points of~$X^\an_0$ whose associated seminorm has no center on~$X$. 

\begin{notation}
Set 
\[X_{\infty} := X^\an_0 - X^\beth.\]
It is an open subset of~$X^\an_{0}$.
\end{notation}

\begin{example}\label{ex:A1Gm}
Consider $\A^1_{k}$ with coordinate~$T$. For $r \in \R_{> 0}$, denote by~$\eta_{r}$ the point of $\E{1}{k_{0}}$ associated to the absolute value
\[ P(T) = \sum_{\ge 0} a_{i} T^i \mapstoo \max_{i\ge 0} (\abs{a_{i}}_{0} r^i).\] 
Note that, for each $P \in k[T] - \{0\}$, we have $\abs{P(\eta_{r})} = r^{-v_{T}(P)}$ if $r \in (0,1)$, and $\abs{P(\eta_{r})} = r^{\deg(P)}$ if $r >1$. 

We have
\[ (\A^1_{k})_{\infty} = \{ x \in \E{1}{k_{0}} : \abs{T(x)} >1\} = \{\eta_{r} : r\in \R_{>1}\}\]
and 
\[ (\G_{m,k})_{\infty} = \{ x \in \E{1}{k_{0}} : 0 < \abs{T(x)} <1\} \cup  \{ x \in \E{1}{k_{0}} : \abs{T(x)} >1\}  = \{\eta_{r} : r\in \R_{>0} - \{1\}\}.\]

Those spaces have a finite number of connected components (1 and 2 respectively), which correspond to the missing points in a smooth compactification. This result easily generalizes to arbitrary curves.
\end{example}

\begin{example}\label{ex:A2}
Consider $\A^2_{k}$ with coordinates $T_{1},T_{2}$. We have
\[ (\A^2_{k})_{\infty} = \{  x \in \E{2}{k_{0}} : \max(\abs{T_{1}(x),\abs{T_{2}(x)}}) > 1\}.\]
\end{example}

The construction $\wc_{\infty}$ behaves well with respect to proper morphisms.

\begin{proposition}\label{prop:finfty}
Let $f\colon X_{1} \to X_{2}$ be a proper morphism between varieties over~$k$. Then $f_{0}^\an$ is proper and we have $(f^\an_0)^{-1}(X_{2}^\beth) = X_{1}^\beth$. 

In particular, $f^\an_0$ restricts to a proper morphism $f_{\infty} \colon (X_{1})_{\infty} \to (X_2)_{\infty}$.
\end{proposition}
\begin{proof}
The morphism $f_{0}^\an$ is proper by \cite[Proposition~3.4.7]{rouge}.

Let $x\in (X_{1})^\an_0$ and set $y:= f(x) \in (X_{2})^\an_0$. The morphism~$f_{0}^\an$ induces an isometry $(\kappa(\rho(y)),\va_{y}) \to (\kappa(\rho(x)),\va_{x})$, and we have a commutative diagram
\[
\begin{tikzcd}
\Spec (\kappa(\rho(x))) \ar{r}{\lambda_{x}} \ar{d}{\varphi} & X_1 \ar{d}{f} \\
\Spec (\kappa(\rho(y))) \ar{r}{\lambda_{y}} & X_2.
\end{tikzcd}
\]
It is clear that, if $x$ belongs to~$X_{1}^\beth$, then $y$ belongs to~$X_{2}^\beth$.

Conversely, assume that $y\in X_{2}^\beth$. Then, the morphism~$\lambda_{y}$ factors through~$\Spec (\kappa(\rho(y))^\circ)$, hence the morphism $f\circ \lambda_{x} = \lambda_{y}\circ \varphi$ factors through~$\Spec (\kappa(\rho(x))^\circ)$. By the valuative criterion of properness, the morphism~$\lambda_{x}$ also factors through~$\Spec (\kappa(\rho(x))^\circ)$, which means that $x\in X_{1}^\beth$.
\end{proof}

As a consequence, the mappings $X \mapsto X_\infty$, $f\mapsto f_\infty$ define a functor from the category of varieties over~$k$ with proper morphisms to that of analytic varieties over~$k_{0}$.

\medbreak

To conclude, let us state the remarkable fact~$X_{\infty}$ may be computed using a compactification of~$X$. 

\begin{proposition}\label{prop:etainfini}
Let $Y$ be a compactification of~$X$, \emph{i.e.} a proper variety in which $X$ embeds as an open subset. Set $Z := Y - X$ and let~$\Yk$ be the completion of~$Y$ along~$Z$. Then, we have a natural identification $\Yk_{\eta} = X_{\infty}$.

In particular, $\Yk_{\eta}$ does not depend on the choice of the compactification~$Y$ of~$X$.
\end{proposition}
\begin{proof}
Let $j\colon \Yk\to Y$ be the canonical morphism. By \cite[Proposition~1.5]{beth}, 
the morphism $j^\beth\colon \Yk^\beth \to Y^\beth = Y^\an_0$ induces an isomorphism between~$\Yk^{\beth}$ and the open subset $r_{Y}^{-1}(Z)$ of~$Y^\an_0$. We identify them in the sequel.

The space $r_{Y}^{-1}(Z)$ consists of the points of $Y^\beth$ whose associate seminorm has a center in~$Z$, that is to say not in~$X$. In other words, $\Yk^{\beth} = Y^\beth - X^\beth$, hence $\Yk_{\eta} = Y^\beth - X^\beth - Z^\beth$.

Since~$Y$ is proper, $Z$ is proper too, and we have $Y^\beth = Y_{0}^\an$ and $Z^\beth = Z_{0}^\an$. It follows that 
\[ \Yk_{\eta} = Y^\an_{0} - X^\beth - Z^\an_{0} = X_{0}^\an - X^\beth = X_{\infty}.\]
\end{proof}

\section{Berkovich compactifications}\label{sec:compactifications}

In this section, we define and study compactifications of analytifications of varieties over~$\hat k$ in the setting of hybrid Berkovich spaces. 

\begin{notation}
For a variety $X$ over~$k$, set 
\[X^+ := X^\hyb - X^\beth.\] 
It is an open subset of~$X^\hyb$, hence an analytic space over~$k_{\hyb}$.
\end{notation}

Note that we have $X^+_{0} = X_{\infty}$ and, for each $\eps \in (0,1]$, $X^+_{\eps} = (X\otimes_{k} \hat k_{\eps})^\an$. 
In particular, we have $X^+_{1} = (X\otimes_{k} \hat k)^\an$.

\begin{proposition}\label{prop:f+}
Let $f\colon X_{1} \to X_{2}$ be a proper morphism of varieties over~$k$. Then the morphism $f^\hyb \colon X_{1}^\hyb \to X_{2}^\hyb$ is proper and restricts to a proper morphism $f^+ \colon X_{1}^+ \to X_{2}^+$.

In particular, the mappings $X \mapsto X^+$, $f\mapsto f^+$ define a functor from the category of varieties over~$k$ with proper morphisms to that of analytic varieties over~$k_\hyb$ with proper morphisms.
\qed
\end{proposition}
\begin{proof}
The properness of $f^\hyb$ follows from~\cite[Proposition~6.5.3]{CTCZ}. By Proposition~\ref{prop:finfty}, we have $(f^\hyb)^{-1}(X_{2}^\beth) = X_{1}^\beth$, hence $f^\hyb$ induces a morphism from $X_{1}^+$ to $X_{2}^+$, which remains proper.
\end{proof}

Since the functor $\wc^+$ comes from analytication, it preserves several properties. 

\begin{lemma}
Let $Y$ be a closed subvariety of~$X$ defined by a sheaf of ideals~$\cI$. Then $Y^+$ identifies to the closed analytic subspace of~$X^+$ defined by the sheaf of ideals~$(\cI^\hyb)_{\vert X^+}$.\qed
\end{lemma}
\begin{proof}
This holds by construction.
\end{proof}

\begin{lemma}\label{lem:XXeps}
Let $X$ be a variety over~$k$. Let $x \in X^\hyb$ and set $\eps(x) := \pr(x)$. The restriction morphism
\[ \cO_{X,x} \to \cO_{X_{\eps(x)},x}\]
is flat.
\end{lemma}
\begin{proof}
The result follows from \cite[Th\'eor\`eme~4.3]{BergerEtale}, noting that the local ring~$\cO_{\cM(k_{\hyb}),\eps(x)}$ is a field.
\end{proof}

We refer to~\cite{BergerEtale} for the definitions of \'etale and smooth morphisms in the setting of Berkovich spaces over Banach rings.

\begin{proposition}\label{prop:propertiesX+}
Let $X$ be a variety over~$k$. If $X$ is (i) Cohen-Macaulay, (ii) $(R_{n})$ for some $n\in \N$, (iii) $(S_{n})$ for some $n\in \N$, (iv) reduced, (v) normal, (vi) regular, (vii) smooth, then so are~$X^\hyb$ and~$X^+$.
\end{proposition}
\begin{proof}
Since $X^+$ is open in~$X^\hyb$, it is enough to prove the result for the latter. 

Let $x\in X^\hyb$ and set $\eps(x) := \pr(x)$. If one of the properties (i) to (vi) holds for $X$, then it holds for the local ring $\cO_{X_{\eps(x)},x}$, which a local ring in a classical analytic space over a valued field (either a complex analytic space, a quotient of such by complex conjugation, or a Berkovich analytic space). We refer to \cite[Exp. XII, Proposition~2.1]{SGA1} for the complex analytic case, and to \cite[Th\'eor\`eme~3.4]{DucrosExcellence} for the case of Berkovich spaces. The result now follows from Lemma~\ref{lem:XXeps}, since the properties may be checked after faithfully flat extension (see for instance \cite[Section~0.5.1]{DucrosExcellence} and the references therein).

The result for property (vii) (smoothness) follows from \cite[Proposition~9.5]{BergerEtale}. \end{proof}

\begin{proposition}\label{prop:propertiesf+}
Let $f\colon X_{1} \to X_{2}$ be a proper morphism of varieties over~$k$. If $f$ is (i) a closed immersion, (ii) surjective, (iii) finite, (iv) flat, (v) \'etale, (vi) smooth, then so are~$f^\hyb$ and~$f^+$.
\end{proposition}
\begin{proof}
The result for properties (i), (ii) and (iii) follows from \cite[Propositions~6.5.3 and~6.6.9]{CTCZ}, and for (v) and (vi) from \cite[Propositions~8.8 and 9.5]{BergerEtale}.

It remains to deal with (iv). Assume that $f$ is flat. Let $x \in X_1^\hyb$ and set $y := f^\hyb(x)$ and $\eps(x):=\pr(x)$. We have a commutative diagram 
\[\begin{tikzcd}
\cO_{X_{1},x} \arrow[r] \arrow[d]& \cO_{X_{1,\eps(x)},x}\arrow[d]\\
\cO_{X_{2},y}\arrow[r] & \cO_{X_{2,\eps(x)},y}\\
\end{tikzcd}.\]
By Lemma~\ref{lem:XXeps}, the horizontal arrows are flat. The right vertical arrow is flat too, by the classical theory over a field (see \cite[Exp. XII, Proposition~3.1]{SGA1} for complex analytic spaces and \cite[Proposition~3.4.6]{rouge} for Berkovich spaces). It follows from \cite[\href{https://stacks.math.columbia.edu/tag/039V}{Lemma 039V}]{stacks-project} that the left arrow is flat too, which proves the claim.  
\end{proof}

For a variety $X$ over~$k$, the spaces $(X\otimes_{k} \hat k_{\eps})^\an$, with $\eps \in (0,1]$, are related by the flow, and all isomorphic to each other. As a result, it looks as if the space~$X^+$ contained several copies of the original space $(X\otimes_{k} \hat k)^\an$. To remedy this situation, we will consider a quotient of the space~$X^+$. This procedure was first carried out by L.~Fantini (over a trivially valued field) in his construction of normalized Berkovich spaces (see~\cite[Section~3]{TXZ}). 

\begin{definition}
Let $X$ be a variety over~$k$. Let $x,y\in X^\hyb$. We say that $x$ and $y$ are \emph{flow-equivalent} if $T(x) = T(y)$. In this case, we write $x \binphi y$.
\end{definition}

\begin{lemma}
Let $X$ be a variety over~$k$. The relation~$\binphi$ on~$X^\hyb$ is an equivalence relation that preserves~$X^+$. 
\end{lemma}
\begin{proof}
 It follows from the definitions that $\binphi$ is an equivalence relation on~$X^\hyb$ (see Lemma~\ref{lem:TxTy}).

For each open affine subset $U = \Spec(A)$, we have
\[ U^\beth = \cM(A) = \{ x\in U^\an : \forall a \in A, \abs{a(x)}\le 1\},\]
hence $U^\beth$ is preserved by~$\binphi$. It follows that~$X^\beth$, and then~$X^+$, are preserved by~$\binphi$ too.
\end{proof}

\begin{definition}
Let $X$ be a variety over $k$. The set
\[X\cp := X^+/\binphi\]
is called the \emph{valuative compactification} of~$X$. We denote by $q \colon X^+ \to X\cp$ the quotient map.

The subset 
\[X^\partial := q(X_{0}^+) = q(X_{\infty})\]
of~$X\cp$ is called the \emph{valuative boundary} of~$X$.

Set 
\[\iota \colon (X\otimes_{k} \hat k)^\an = X^+_{1} \lhook\joinrel\longrightarrow X^+ \xrightarrow[]{\ q\ } X\cp.\]
\end{definition}

\begin{remark}
The space $X^\partial = X_{\infty}/\binphi$ is a typical example of a normalized Berkovich space in the sense of L.~Fantini (see \cite{TXZ}). 
\end{remark}

\begin{lemma}\label{lem:iota}
Let $X$ be a variety over $k$. The map~$\iota$ 
induces a bijection onto $q(X^+_{>0})$. 
\end{lemma}
\begin{proof}
For each $x \in X^+_{>0} = X^\hyb_{>0}$, 
we have $T(x) \cap X^+_{1} = \{x^{1/\pr(x)}\}$. The result follows.
\end{proof}

We abusively denote by~$X^\an$ the analytic space $(X\otimes_{k} \hat k)^\an$. From now on, we identify it to its image by~$\iota$. (We will see later that~$\iota$ is an open map, and even an open immersion of locally ringed spaces, see Lemmas~\ref{lem:qopen} and Proposition~\ref{prop:iotaiso}.) We then have 
\[X\cp = X^\an \sqcup X^\partial.\]

\begin{example}\label{ex:A1partial}
Consider $\A^1_{k}$ with coordinate~$T$. It follows from Example~\ref{ex:A1Gm} that the valuative boundary $(\A^1_{k})^\partial$ consists of one point, say $\ast_{\infty}$, with residue field~$k(\!(T^{-1})\!)$. The valuative boundary $(\G_{m,k})^\partial$ contains an additional point, say $\ast_{0}$, with residue field~$k(\!(T)\!)$.
\end{example}

\begin{example}\label{ex:A2partial}
Consider $\A^2_{k}$ with coordinates $T_{1},T_{2}$. Recall that we have
\[ (\A^2_{k})_{\infty} = \{  x \in \E{2}{k_{0}} : \max(\abs{T_{1}(x),\abs{T_{2}(x)}}) > 1\}\]
(see Example~\ref{ex:A2}) and let us now consider the quotient by the flow. For each $r\in \R_{>1}$, we have a bijection (and even a homeomorphism with respect to the quotient topology)
\[(\A^2_{k})_{\infty}/\binphi \simeq \{  x \in \E{2}{k_{0}} : \max(\abs{T_{1}(x),\abs{T_{2}(x)}}) = r\}.\]
We recover this way the valuative tree introduced by Ch.~Favre and M.~Jonsson (see \cite{ValuativeTree}).
\end{example}

\begin{remark}
Let $X$ be a proper variety over~$k$ and let $f \colon X \to \G_{m,k}$ be a proper morphism. It follows from Proposition~\ref{prop:f+} that we have an induced morphism $f\cp \colon X\cp \to \G_{m,k}\cp$. Let us describe~$f\cp$ as a fibration over $\G_{m,k}\cp = \G_{m,k}^\an \cup \{\ast_{0}\} \cup \{\ast_{\infty}\}$. 

\begin{itemize}
\item Over the open subset $\G_{m,k}^\an$, the morphism $f\cp$ identifies to the classical analytification morphism $f^\an \colon X^\an \to \G_{m,k}^\an$. 

\item Over the point~$\ast_{0}$, the fiber of~$f\cp$ identifies to the analytification of the variety $X \times_{\G_{m,k}} \Spec(k(\!(T)\!))$ over~$k(\!(T)\!)$\footnote{The analytification procedure requires a normalisation of the absolute value on $k(\!(T)\!)$, but all choices give rise to equivalent spaces under the flow.}, the morphism $\Spec(k(\!(T)\!)) \to \G_{m,k}$ being induced by the inclusion $k[T,T^{-1}] \subset k(\!(T)\!)$. 

\item Over the point~$\ast_{\infty}$, the fiber of~$f\cp$ identifies to the analytification of $X \times_{\G_{m,k}} \Spec(k(\!(T^{-1})\!))$ over~$k(\!(T^{-1})\!)$, with a base change map defined as above.
\end{itemize}
We recover in this way the construction of S.~Boucksom and M.~Jonsson in~\cite{BJ} (see in particular the appendix). Starting with a meromorphic family~$X$ over the complex punctured disk~$\D_{\C}^\ast$, they construct a hybrid space $X^\textrm{BJ}$ by adding a non-Archimedean fiber $(X \times_{\O(\D_{\C}^\ast)} \C(\!(T)\!))^\an$ over the puncture. Concretely, the space $X^\textrm{BJ}$ is realized as the analytification of~$X$ over the Banach ring
\[ A_{r} := \big\{\sum_{i\in \Z} a_{i} t^i \in \C(\!(T)\!) :  \sum_{i\in \Z} \norm{a_{i}}_{\hyb}\, r^i < \infty\big\},\] 
endowed with the norm defined by $\norm{\sum_{i\in \Z} a_{i} t^i} = \sum_{i\in \Z} \norm{a_{i}}_{\hyb}\, r^i$, for some $r\in (0,1)$.

The relationship with our construction comes from the fact that the Berkovich spectrum~$\cM(A_{r})$ of~$A_{r}$ identifies to the hybrid circle 
\[ C_{\hyb}(r) := \{x \in \Ahyb{1}{\C} : \abs{T(x)} = r\},\]
which again identifies to the quotient by the flow of the hybrid punctured disk
\[ \D_{\hyb}^\ast := \{x \in \Ahyb{1}{\C} : 0 < \abs{T(x)} < 1\}.\]
Writing $\D_{\hyb}^\ast = (\D_{\hyb}^\ast \cap \pr^{-1}((0,1])) \cup  (\D_{\hyb}^\ast \cap \pr^{-1}(0))$ and passing to the quotient, we get
\[  C_{\hyb}(r) = \D_{\hyb}^\ast/\binphi = \D_{\C}^\ast \cup \{\ast_{0}\}  \subset \G_{m,\C}\cp.\]
\end{remark}

\subsection{Topological properties}\label{sec:topprop}

Let $X$ be a variety over~$k$. We endow $X\cp := X^+/\binphi$ with the quotient topology, and investigate its topological properties.

\begin{lemma}\label{lem:qopen}
The maps $q \colon X^+ \to X\cp$ and $\iota \colon X^\an \to X\cp$ are open. 

In particular, $\iota$ induces a homeomorphism from~$X^\an$ onto its image (which is open in~$X\cp$).
\end{lemma}
\begin{proof}
Let $U$ be an open subset of~$X^+$. Its image in $X\cp = X^+/\binphi$ is open if, and only if, $T(U)$ is open in~$X^+$. The latter property follows from Proposition~\ref{prop:Phicontinuous}, hence $q$ is open.

\medbreak

In order to prove that $\iota$ is open, it is enough to show that each $x\in X^\an$ admits a basis of open neighborhoods~$\Vc_{x}$ with the following property: for each $V\in \Vc_{x}$, there exists an open subset~$V'$ of~$X^+$ such that $V'=V \mod \binphi$. Indeed, in this case, we have $\iota(V) = q(V')$, which is open, by the first part.

Let $x\in X^\an$. We may assume that $X$ is affine. Let us consider a closed embedding $j \colon X \to \A^n_{k}$ and fix coordinates $T=(T_{1},\dotsc,T_{n})$ on~$\A^n_{k}$.

By definition of the topology, $j(x)$ admits a basis of neighborhoods in $\E{n}{\hat k}$ of the form 
\[ W = \bigcap_{i=1}^p \big\{ y \in \E{n}{\hat k} : r_{i} < \abs{P_{i}(y)} < s_{i}\big\},\]
with $P_{1},\dotsc,P_{p} \in k[T]$, $r_{1},\dotsc,r_{p},s_{1},\dotsc,s_{p} \in \R$, and $r_{i} < \abs{P_{i}(x)} < s_{i}$ for each $i \in \{1,\dotsc,p\}$. 

Since $k$ is not trivially valued, there exists $\alpha \in k$ such that $0 < \abs{\alpha} < 1$. Up to multiplying the $P_{i}$'s by some power of~$\alpha$, and the corresponding~$r_{i}$ and~$s_{i}$ by the same power of~$\abs{\alpha}$, we may assume that $r_{1},\dotsc,r_{p},s_{1},\dotsc,s_{p} \in (-1,1)$. In this case, for each $i\in \{1,\dotsc,p\}$, there exists $\sigma_{i} \in \{-1,1\}$, $\beta_{i},\gamma_{i} \in \{0,\alpha\}$ and $u_{i},v_{i} \in \R_{>0}$ such that $r_{i} = \sigma_{i} \abs{\beta_{i}}^{u_{i}}$ and $s_{i} = \abs{\gamma_{i}}^{v_{i}}$, so that we have
\[ W = \bigcap_{i=1}^p \big\{ y \in \E{n}{\hat k} : \sigma_{i} \abs{\beta_{i}}^{u_{i}} < \abs{P_{i}(y)} < \abs{\gamma_{i}}^{v_{i}}\big\}.\]

Set 
\[ W' := \bigcap_{i=1}^p \big\{ y \in \Ahyb{n}{\hat k, >0} : \sigma_{i} \abs{\beta_{i}}^{u_{i}} < \abs{P_{i}(y)} < \abs{\gamma_{i}}^{v_{i}}\big\}.\]
It is an open subset of $\Ahyb{n}{\hat k, >0}$ and we have $W' = W \mod \binphi$.

We obtain the desired result by pulling back by~$j$.

\medbreak

The last part of the statement now follows from Lemma~\ref{lem:iota}. 
\end{proof}

\begin{proposition}\label{prop:Hausdorff}
The space $X\cp$ is Hausdorff.
\end{proposition}
\begin{proof}
We want to prove that $X\cp$ is Hausdorff, which amounts to proving that the graph~$\Gamma_{\binphi,X^+}$ of~$\binphi$ is closed in $X^+\times X^+$. 

By Chow's lemma, we may find a quasi-projective variety~$Y$ and a proper surjective morphism $f \colon Y \to X$. By \cite[Proposition~6.5.3]{CTCZ}, the induced morphism $f^\hyb \colon Y^\hyb \to X^\hyb$ is proper and surjective too, hence so is $f^+ \colon Y^+ \to X^+$, by Proposition~\ref{prop:finfty}. The induced morphism $f^+\times f^+ \colon Y^+ \times Y^+ \to X^+\times X^+$ between the Cartesian products is still surjective and proper (as a topological morphism), hence closed. It follows from the definitions that $\Gamma_{\binphi,X^+} = (f^+\times f^+)(\Gamma_{\binphi,Y^+})$. In particular, if $\Gamma_{\binphi,Y^+}$ is closed in $Y^+\times Y^+$, then $\Gamma_{\binphi,X^+}$ is closed in $X^+\times X^+$.

It remains to show that $\Gamma_{\binphi,Y^+}$ is closed in $Y^+\times Y^+$. Let $(x,y) \in (Y^+\times Y^+) \setminus \Gamma_{\Phi}$. Since $Y$ is quasi-projective, there exists an open affine subset $U = \Spec(A)$ of~$Y$ such that~$x$ and~$y$ belong to~$U^+$. Recall that, for each $k_\hyb$-analytic space~$\Xc$, we have a projection map $\pr \colon \Xc \to \cM(k_{\hyb}) = [0,1]$ (see Section~\ref{sec:defhybrid}).

\smallbreak

$\bullet$ Assume that $\pr(x) >0$ and $\pr(y)>0$. 

We have $x^{1/\pr(x)} \ne y^{1/\pr(y)}$ in $U^\an$. Up to switching~$x$ and~$y$, we may assume that there exists $a\in A$ such that $\abs{a(x^{1/\pr(x)})} < \abs{a(y^{1/\pr(y)})}$. Since $k$ is not trivially valued, up to replacing~$a$ by an element of the form $\lambda a^n$, with $\lambda\in k^\ast$ and $n \in \N^\ast$, we may assume that  $\abs{a(x^{1/\pr(x)})}<1$ and $\abs{a(y^{1/\pr(y)})}>1$. Then the set 
\[W := \{(u,v) \in U^+\times U^+ \colon \abs{a(u)}<1, \abs{a(v)}>1\}\] 
is a neighborhood of~$(x,y)$ in $U^+\times U^+$ that does not meet~$\Gamma_{\binphi}$.

\smallbreak

$\bullet$ Assume that $\pr(x)=0$ and $\pr(y) >0$, or that $\pr(x)>0$ and $\pr(y) =0$.

Up to switching~$x$ and~$y$, we may assume that $\pr(x)=0$ and $\pr(y) >0$. Let $\alpha_{1} < \alpha_{2}\in (0,\pr(y))$. Then the set 
\[W_{-} := \{(u,v) \in U^+\times U^+ \colon \pr(u)<\alpha_{1}, \pr(v)>\alpha_{2}\}\] 
is a neighborhood of~$(x,y)$ in $U^+\times U^+$ that contains no points of the form $(z,z^\beta)$ with $z\in U^+$ and $\beta\in I^\ast_{z} \cap (0,\alpha_{2}/\alpha_{1}]$. 

Since $x\notin U^\beth$, there exists $a\in A$ such that $|a(x)|>1$. Let $r \in (1,|a(x)|)$ and $s\in (|a(y)|,+\infty)$. Then the set 
\[W_{+}:=\{(u,v) \in U^+\times U^+ \colon \abs{a(u)}>r, \abs{a(v)}<s\}\] 
is a neighborhood of~$(x,y)$ in $U^+\times U^+$ that contains no points of the form $(z,z^\beta)$ with $z\in U^+$ and $\beta\in I^\ast_{z} \cap [\log(s)/\log(r),+\infty)$. 

If $\log(s)/\log(r) \le \alpha_{2}/\alpha_{1}$, then $W_{-}\cap W_{+}$ is a neighborhood of~$(x,y)$ in $U^+\times U^+$ that does not meet~$\Gamma_{\binphi}$. Otherwise, for each $\alpha\in [\alpha_{2}/\alpha_{1},\log(s)/\log(r)]$, by assumption, there exists $a\in A$ such that $\abs{a(y)} \ne \abs{a(x)}^\alpha$. It follows that there exist a neighborhood~$W_{\alpha}$ of $(x,y)$ in $U^+\times U^+$ and a neighborhood~$J_{\alpha}$ of~$\alpha$ in~$\R_{>0}$ such that, for each $(u,v) \in W_{\alpha}$ and $\beta\in J_{\alpha}$, we have $\abs{a(v)} \ne \abs{a(u)}^\beta$. A compactness argument shows that there exists a finite subset~$S$ of~$[\alpha_{2}/\alpha_{1},\log(s)/\log(r)]$ such that $W_{-}\cap W_{+}\cap \bigcap_{\alpha \in S} W_{\alpha}$ is a neighborhood of~$(x,y)$ in $U^+\times U^+$ that does not meet~$\Gamma_{\binphi}$.

\smallbreak

$\bullet$ Assume that $\pr(x) = \pr(y) = 0$. 

Since $x\notin U^\beth$, as before, there exist $\alpha_{+} \in \R_{>0}$ and a neighborhood~$W_{+}$ of~$(x,y)$ in $U^+\times U^+$ that contains no points of the form $(z,z^\beta)$ with $z\in U^+$ and $\beta\in I^\ast_{z} \cap [\alpha_{+},+\infty)$. 

Since $y\notin U^\beth$, similarly, we prove that there exist $\alpha_{-} \in \R_{>0}$ and a neighborhood~$W_{-}$ of~$(x,y)$ in $U^+\times U^+$ that contains no points of the form $(z,z^\beta)$ with $z\in U^+$ and $\beta\in I^\ast_{z} \cap (0,\alpha_{-}]$. 

The result then follows from a compactness argument as in the previous case.
\end{proof}

\begin{corollary}\label{cor:locH}
The space~$X\cp$ is locally compact.
\end{corollary}
\begin{proof}
The space~$X^\hyb$ is locally compact, hence so is~$X^+$. Since $q$ is open and $X\cp$ is Hausdorff, we deduce that $X\cp$ is locally quasi-compact. 
\end{proof}

\begin{lemma}\label{lem:trivialfibercompact}
The space $X^\partial$ is compact.
\end{lemma}
\begin{proof}
Let $X$ be a variety over $k$. Let~$Y$ be a compactification of~$X$ and set $Z := Y-X$. Up to blowing up~$Z$ in~$X$, we may assume that~$Z$ is a Cartier divisor. Let~$\cI$ be the sheaf of ideals defining~$Z$ in~$Y$. Let $(U_{i})_{i\in I}$ be a finite cover of~$Y$ by affine open subsets such that, for each $i\in I$, $\cI_{\vert U_{i}\cap Z}$ is generated by a single element $h_{i} \in \cO(U_{i})$.

Fix $r \in (0,1)$. Let $i\in I$. Denoting by~$\Uk_{i}$ the completion of~$U_{i}$ along~$U_{i}\cap Z$, we have 
\[\Uk_{i,\eta} = \{x\in U_{i}^\beth \colon 0 < \abs{h_{i}(x)} < 1\}.\]  
The quotient $\Uk_{i,\eta}/\binphi$ is compact, since it is the image of the compact set $\{x\in U_{i}^\beth \colon |h_{i}(x)| = r\}$.

By Proposition~\ref{prop:etainfini}, we have $X_{\infty} = \Yk_{\eta}$, hence $X^\partial = \Yk_{\eta}/\binphi$. The result follows, since $\Yk_{\eta}$ is covered by the $\Uk_{i,\eta}$'s, and $X^\partial$ is Hausdorff, by Proposition~\ref{prop:Hausdorff}.
\end{proof}

\begin{theorem}\label{th:compact}
The space~$X\cp$ is compact.
\end{theorem}
\begin{proof}
Let $X$ be a variety over $k$. As in the previous proof, let~$Y$ be a compactification of~$X$ such that $Z := Y-X$ is a Cartier divisor. Let~$\cI$ be the sheaf of ideals defining~$Z$ in~$Y$. Fix $r \in (0,1)$. 

Recall that $X\cp = X^\partial \cup X^\an$, where $X^\an$ is identified to its image by~$\iota$. By Proposition~\ref{prop:Hausdorff}, $X\cp$ is Hausdorff and, by Lemma~\ref{lem:trivialfibercompact}, $X^\partial$ is compact, so it is enough to prove that $X^\an$ is contained in a compact subset of~$X\cp$.

Let $x\in Z_{1}^\an$. Choose an affine open subset $U = \Spec(A)$ of~$Y$ such that $U_{1}^\an$ contains~$x$. We may assume that $A$ is a $k$-algebra of finite type, say $A = k[t_{1},\dotsc,t_{n}]/J$. Up to shrinking~$U$, we may assume that~$\cI_{Z\cap U}$ is generated by a single element $h \in A$. Since $k$ is not trivially valued, up to rescaling the $t_{i}$'s, we may assume that, for each~$i$, we have $\abs{t_{i}(x)}< 1$. Set 
\[V_{x} := \{y\in U_{1}^\an \colon \forall i=1,\dotsc,n,\  \abs{t_{i}(y)}< 1, \abs{h(y)}< r\}.\]
It is an open neighborhood of~$x$ in~$Y_{1}^\an$. 

Consider the set 
\[K_{x} := \{z\in U^\hyb \colon \forall i=1,\dotsc,n,\ \abs{t_{i}(z)}\le 1,\ \abs{h(z)}=r\}.\]
It is a compact subset of~$X^\hyb$, and it is contained in~$X^+$, by Proposition~\ref{prop:etainfini}. 

For each $y\in V_{x} \cap X_{1}^\an$, we have $\abs{h(y)} \in (0,r)$, hence there exists $\alpha \in (0,1)$ such that $|h(y)|^\alpha = r$. It follows that $q(V_{x} \cap X_{1}^\an)$ is contained in~$q(K_{x})$. 

Since~$Z_{1}^\an$ is compact, there exists a finite subset~$F$ of~$Z_{1}^\an$ such that the open subset $V := \bigcup_{x\in F} V_{x}$ of~$Y_{1}^\an$ covers~$Z_{1}^\an$. In particular, $W := Y_{1}^\an- V = X_{1}^\an - (V\cap X_{1}^\an)$ is a compact subset of~$X_{1}^\an$. We conclude by writing 
\[X^\an \simeq q(X_{1}^\an) = q(W) \cup q(V\cap X_{1}^\an) \subseteq q(W) \cup \bigcup_{x\in F} q(K_{x}).\]
\end{proof}

\begin{lemma}
Assume that $k$ is countable. Then the spaces $X^\hyb$, $X^+$ and~$X\cp$ are metrizable. 
\end{lemma}
\begin{proof}
The spaces $X^\hyb$, $X^+$ and $X\cp$ are Hausdorff and locally compact. By Urysohn's metrization theorem, it suffices to show that they are second-countable. Since $X^+$ is open in $X^\hyb$ and $X\cp$ is a quotient of~$X^+$, we only need to prove that $X^\hyb$ is second-countable.

Since~$X$ is quasi-compact, we may reduce to the case where~$X$ is affine, and then to the case where $X = \A^N_{k}$, for some $N\in \N$. Let us fix coordinates $T_{1},\dotsc,T_{N}$ on $\A^N_{k}$. Then, the topology of $\Ahyb{N}{k}$ is generated by the countably many open sets of the form 
\[ \{ x \in \Ahyb{N}{k} : r < \abs{P(x)} < s\},\]
for $P \in k[T_{1},\dotsc,T_{n}]$, $r,s \in \Q$. The result follows.
\end{proof}

\begin{remark}
If $k$ is not countable, then $\Ahyb{1}{k}$ is not metrizable. Indeed, in this case, the Gauss point of $(\Ahyb{1}{k})_{0} = \E{1}{k_{0}}$ admits no countable basis of neighborhoods.
\end{remark}

\begin{proposition}\label{prop:Xandense}
The space $X^\hyb_{>0}$ is dense in~$X^\hyb$. In particular, $X^\an$ is dense in~$X\cp$. 
\end{proposition}
\begin{proof}
If $X^\hyb_{>0}$ were not dense in~$X^\hyb$, there would exist a non-empty open subset~$U$ of $X^\hyb$ that is contained in~$X^\hyb_{0}$. By \cite[Proposition~6.4.1]{CTCZ}, the projection map $\pr \colon U \to \cM(k_{\hyb})$ is open, which yields a contradiction. This proves the first part of the statement. The second part follows by applying~$q$.
\end{proof}

\begin{proposition}
The space $X\cp$ is locally path-connected. 

If $X$ is connected, then $X\cp$ is path-connected.
\end{proposition}
\begin{proof}
By \cite[Th\'eor\`eme~7.2.17]{CTCZ}, $X^+$ is locally path-connected. By Lemma~\ref{lem:qopen}, $q$ is open, and the result follows.

Assume that $X$ is connected. Then, $X^\an$ is connected too. It then follows from Proposition~\ref{prop:Xandense} that $X\cp$ is connected, hence path-connected by the first part of the statement.
\end{proof}

\subsection{Sheaf-theoretic properties}

Let $X$ be a variety over~$k$. We endow $X\cp$ with the sheaf $\O_{X\cp} := q_{\ast} \O_{X^+}$. 

\begin{remark}
Let $U$ be an open subset of~$X\cp$ and let $x\in U$. Since we have quotiented by the flow, for $f\in \O_{X\cp}(U) = \cO_{X^+}(q^{-1}(U))$, the absolute value of~$f$ at~$x$ is no longer well-defined. 

However, properties such as $f(x)=0$, $f(x)\ne 0$, $\abs{f(x)}<1$, $\abs{f(x)}=1$, $\abs{f(x)} \le \abs{g(x)}$ for $g\in \O_{X\cp}(U)$, \dots still make sense, since they are invariant under raising to a non-zero power. We will allow us to use them in the sequel.
\end{remark}

\begin{lemma}
Let $x\in X\cp$. The ring~$\cO_{X\cp,x}$ is a local ring with maximal ideal
\[ \m_{x} = \{ f\in \cO_{X\cp,x} : f(x)=0\}.\]
In particular, $X\cp$ is a locally ringed space.
\qed
\end{lemma}

\begin{proposition}\label{prop:functorcp}
Let $f\colon X_{1} \to X_{2}$ be a proper morphism of varieties over~$k$. Then the morphism $f^+ \colon X_{1}^+ \to X_{2}^+$ from Proposition~\ref{prop:f+} induces a morphism of locally ringed spaces $f\cp \colon X_{1}\cp \to X_{2}\cp$.

In particular, the mappings $X \mapsto X\cp$, $f\mapsto f\cp$ define a functor from the category of varieties over~$k$ with proper morphisms to that of locally ringed spaces.
\qed
\end{proposition}

\begin{proposition}\label{prop:iotaiso}
The map $\iota \colon X^\an \to X\cp$ is an open immersion. 
\end{proposition}
\begin{proof}
By Lemmas~\ref{lem:iota} and~\ref{lem:qopen}, $\iota$ induces a homeomorphism from~$X^\an$ onto~$q(X^\hyb_{>0})$.

It remains to prove that, for each open subset~$U$ of~$q(X^\hyb_{>0})$, the natural morphism $\cO_{X\cp}(U) \to \cO_{X^\an}(\iota^{-1}(U))$ is an isomorphism. Since $\iota^{-1}(U) = q^{-1}(U) \cap X^\hyb_{1}$ and $T(q^{-1}(U)) = q^{-1}(U)$, the result follows from Lemmas~\ref{lem:>0eps} and~\ref{lem:hybaneps}.
\end{proof}

Let us now compare coherent sheaves on~$X^+$ and~$X\cp$.

\begin{proposition}\label{prop:q-1q*}
Let $U$ be an open subset of~$X^+$ such that $q^{-1}(q(U)) = U$. For each coherent sheaf~$\cF$ on~$U$, the natural morphism $q^{-1} q_{\ast} \cF \to \cF$ is an isomorphism. Moreover, the functor $q_{\ast}$ is exact on coherent sheaves on~$U$.

In particular, the natural morphism $q^{-1} \cO_{X\cp} \to \cO_{X^+}$ is an isomorphism.

\medbreak

Let $V$ be an open subset of~$X\cp$. For each coherent sheaf~$\cG$ on~$V$, the natural morphism $\cG \to q_{\ast} q^{-1} \cG$ is an isomorphism.
\end{proposition}
\begin{proof}
Let $\cF$ be a coherent sheaf on~$U$. Let $x\in U$. We have 
\[ (q^{-1} q_{\ast} \cF)_{x} = (q_{\ast}\cF)_{q(x)} = \colim_{W \ni q(x)} \cF(q^{-1}(W)),\]
where $W$ runs through the open neighborhoods of~$q(x)$ in $q(U)$. 

Let $s \in (q_{\ast} \cF)_{q(x)}$ whose image in~$\cF_{x}$ is~0. There exists an open neighborhood~$V$ of~$x$ in~$U$ such that $s_{\vert V} = 0$. By Theorem~\ref{th:coherentflow}, it follows that~$s$ is~0 on $T(V) = q^{-1}(q(V))$. By Lemma~\ref{lem:qopen}, $q(V)$ is open, hence $s = 0$ in $(q_{\ast} \cF)_{q(x)}$. We deduce that the natural map $(q^{-1} q_{\ast} \cF)_{x} \to \cF_{x}$ is injective. 

By similar arguments, we prove that the natural map $(q^{-1} q_{\ast} \cF)_{x} \to \cF_{x}$ is surjective. The first statement follows. The above description of the stalks ensures the exactness of~$q_{\ast}$.

\medbreak

Let $\cG$ be a coherent sheaf on~$V$. The statement being local, we may assume that we have an exact sequence $\cO_{X\cp}^m \to \cO_{X\cp}^n \to \cG \to 0$. Applying the functor $q^{-1}$ and using the first part of the statement, we get an exact sequence $\cO_{X^+}^m \to \cO_{X^+}^n \to q^{-1}\cG \to 0$. It follows from this exact sequence and that $q^{-1}\cG$ is coherent. By the first part of the statement, we finally get a commutative diagram with exact rows
\[
\begin{tikzcd}
\cO_{X\cp}^m \arrow[r] \arrow[d,equal]& \cO_{X\cp}^n \arrow[r] \arrow[d, equal]& \cG \arrow[r] \arrow[d]& 0\\ 
\cO_{X\cp}^m \arrow[r] & \cO_{X\cp}^n \arrow[r] & q_{\ast}q^{-1}\cG \arrow[r] & 0\\ 
\end{tikzcd}
\]
the result follows.
\end{proof}

\begin{corollary}
The sheaf $\cO_{X\cp}$ is coherent.
\end{corollary}
\begin{proof}
This follows from the exactness of~$q^{-1}$ and~$q_{\ast}$ and the coherence of~$\cO_{X^+}$, by similar arguments as in the proof of Proposition~\ref{prop:q-1q*}.
\end{proof}

\begin{corollary}\label{cor:eqX+Xcp}
For each open subset $V$ of~$X\cp$, the functor~$q^{-1}$ induces an equivalence between the categories of coherent sheaves on~$V$ and on~$q^{-1}(V)$, with quasi-inverse~$q_{\ast}$.

In particular, the categories of coherent sheaves on~$X^+$ and~$X\cp$ are equivalent.
\qed
\end{corollary}

Let us finally show that the functor $X\mapsto X^+$ preserves many properties.

\begin{corollary}\label{cor:propertiesXcp}
Let $X$ be a variety over~$k$. If $X$ is (i) Cohen-Macaulay, (ii) $(R_{n})$ for some $n\in \N$, (iii) $(S_{n})$ for some $n\in \N$, (iv) reduced, (v) normal, (vi) regular, then so is~$X\cp$.
\end{corollary}
\begin{proof}
By Proposition~\ref{prop:q-1q*}, for each $x\in X^+$, we have an isomorphism $\cO_{X\cp,q(x)} \simeq \cO_{X^+,x}$. Since the properties of the statement are properties of local rings, the result follows from Proposition~\ref{prop:propertiesX+}.
\end{proof}

\begin{corollary}\label{cor:propertiesfcp}
Let $f\colon X_{1} \to X_{2}$ be a proper morphism of varieties over~$k$. If $f$ is (i) a closed immersion, (ii) surjective, (iii) finite, (iv) flat, (v) \'etale, then so is~$f\cp$.
\end{corollary}
\begin{proof}
The result follows from Proposition~\ref{prop:propertiesf+}, as in the proof of Corollary~\ref{cor:propertiesXcp}.
\end{proof}

\begin{remark}
We did not include the smoothness property in Corollaries~\ref{cor:propertiesXcp} and~\ref{cor:propertiesfcp}, since it would require to define smoothness in the setting of compactifications. However, once this definition is given, the result should certainly hold true and be a direct consequence of Propositions~\ref{prop:propertiesX+} and~\ref{prop:propertiesf+}, as for the other properties.
\end{remark}

\section{GAGA theorems}\label{sec:GAGA}

We investigate coherent sheaves on~$X^+$ and~$X\cp$ and prove GAGA theorems relating them to coherent sheaves on~$X$. 

To do so, we consider~$X^+$ from the extrinsic point of view, that is to say as a subset of~$Y^\hyb$, where $Y$ is an algebraic compactification of~$X$, and eventually rely on the usual GAGA theorems on~$Y$. Consequently, our main technical results are extension results from $X^+$ to $Y^\hyb_{>0}$, either for sections of coherent sheaves (that extend up to some meromorphic singularities, see the final part of Theorem~\ref{th:VYZStein} and Lemma~\ref{lem:globalsectionsmeromorphic}), or for the sheaves themselves (see Theorem~\ref{th:extension}).

\subsection{Valuative thickenings}

Let $X$ be a variety over~$k$. We have seen in Proposition~\ref{prop:etainfini} 
that the spaces~$X^\beth$ and~$X_{\infty}$ may be defined using an extra ambient variety~$Y$ and a closed subvariety~$Z$ of it with $X = Y-Z$. In this section, we will work directly with~$Y$ and~$Z$, forgetting~$X$ for a while. We will define some sort of hybrid thickenings of~$Z$ in~$Y$, taking the form of subsets of~$Y^\hyb$ and study their properties.

\begin{notation}\label{nota:VYZ}
Let $Y$ be a variety over~$k$ and $Z$ be a closed subvariety of~$Y$. Consider the following subsets of $Y_0^\an = Y_{0}^\hyb \subseteq Y^\hyb$:
\[\begin{cases}
V^\rho_{Y}(Z) := \rho_{Y}^{-1}(Z);\\
V^r_{Y}(Z) := r_{Y}^{-1}(Z); \\  
V^0_{Y}(Z) :=r_{Y}^{-1}(Z) \cap \rho_{Y}^{-1}(Z)  = Y^\beth \cap \rho_{Y}^{-1}(Z);\\ 
V^\ast_{Y}(Z) := r_{Y}^{-1}(Z) - \rho_{Y}^{-1}(Z).
\end{cases}\]
We endow them with the overconvergent structures inherited from~$Y^\hyb$. 
\end{notation}

Note that, since $r_{Y}$ is anticontinuous and $\rho_{Y}$ is continuous, $V^r_{Y}(Z)$ and $V^\ast_{Y}(Z)$ are open in~$Y^\beth$, and $V^\rho_{Y}(Z)$ and $V^0_{Y}(Z)$ are closed in~$Y^\hyb$. 

Remark that $V^\rho_{Y}(Z)$ identifies to~$Z_{0}^\an$ as a topological space (but their structure sheaves differ). Using the terminology from Definition~\ref{def:center} and Lemma~\ref{lem:center}, we may characterize $V^r_{Y}(Z)$ (resp. $V^\ast_{Y}(Z)$) as the set of points in~$Y_{0}^\an$ (resp. $X_{0}^\an$) that have a center on~$Z$. 


\begin{lemma}\label{lem:unionintersection}
Let $(U_{i})_{i\in I}$ be a family of open subsets of~$Y$. Set $U_{\cup} := \bigcup_{i\in I} U_{i}$ and $U_{\cap} := \bigcap_{i\in I} U_{i}$.
For $\square \in \{\rho, r , 0, \ast\}$, we have 
\[ V^\square_{U_{\cup}}(Z\cap U_{\cup}) = \bigcup_{i\in I} V^\square_{U_{i}}(Z\cap U_{i})\]
and 
\[ V^\square_{U_{\cap}}(Z\cap U_{\cap}) = \bigcap_{i\in I} V^\square_{U_{i}}(Z\cap U_{i}).\]
\end{lemma}
\begin{proof}
Let $x \in Y_{0}^\an$. Assume that $x$ has a center on~$Y$. If there exists a closed subvariety~$F$ of~$Y$ such that $x$ belongs to~$F_{0}^\an$, then its center belongs to~$F$. By contrapositive, if $x$ has a center on~$U_{i}$, for some $i \in I$, then $x$ belongs to~$(U_{i})_{0}^\an$. 

The result follows from this remark, using the interpretations of the sets in terms of centers.
\end{proof}

\subsubsection{The affine case}

Here, we focus on the case of affine varieties. Let $Y = \Spec(A)$ be an affine variety over~$k$ and let $Z = V(I)$ be a subvariety of~$Y$.

Let us fix some notation. Let $A = k[T_{1},\dotsc,T_{n}]/J$ be a presentation of the ring~$A$. Let $h = (h_{1},\dotsc,h_{m})$ be a family of generators  of~$I$.

Consider the affine space $\Ahyb{n}{k}$ 
with coordinates $T_{1},\dotsc,T_{n}$. The space $Y^\hyb$ may be identified to the closed analytic subset of $\Ahyb{n}{k}$ defined by the sheaf of ideals generated by~$J$.

We may further identify the various spaces in play to some concrete spaces:
\[Y^\hyb_{0} = \{x \in \big(\Ahyb{n}{k}\big)_{0} : \forall f \in J, f(x)=0\},\]
\[Y^\beth = \{x \in Y^\hyb_{0} :  \forall i=1,\dotsc,n, \abs{T_{i}(x)}\le 1\},\]
\[V_{Y}^\rho(Z) = \{x \in Y^\hyb_{0} : \forall j=1,\dotsc,m, h_{j}(x) = 0\},\]
\[V_{Y}^r(Z) = \{x \in Y^\beth : \forall j=1,\dotsc,m, \abs{h_{j}(x)} <1\},\]
\[V^0_{Y}(Z) = \{x \in Y^\beth :  \forall j=1,\dotsc,m, h_{j}(x) = 0\},\]
\[V^\ast_{Y}(Z) = \{x \in Y^\beth :  \forall j=1,\dotsc,m, 0 < \abs{h_{j}(x)} < 1\}.\]
Recall that we endow the spaces $V_{Y}^\square(Z)$ with the overconvergent structures inherited from $Y^\hyb$. 

\begin{proposition}\label{prop:isoVV0affine}
For each coherent sheaf~$\cF$ on~$V^\rho_{Y}(Z)$ (resp. $\cF(V_{Y}^r(Z))$), the restriction map
\[ \cF(V_{Y}^\rho(Z)) \too \cF(V_{Y}^0(Z)) \textrm{ (resp. } \cF(V_{Y}^r(Z)) \too \cF(V_{Y}^0(Z))\, ) \]
is an isomorphism.

Moreover, the restriction map induces an equivalence between the categories of coherent sheaves on~$V_{Y}^\rho(Z)$ (resp. $V_{Y}^r(Z)$) and~$V^0_{Y}(Z)$.
\end{proposition}
\begin{proof}
For $t>1$, set 
\[ V^{\rho,t} := \{x \in Y_{0^\dag}^\hyb : \forall i=1,\dotsc,n, \abs{T_{i}(x)}\le t,\ \forall j=1,\dotsc,m, h_{j}(x) = 0\}.\]
Note that, for each $t > 1$, we have $T(V^{\rho,t}) = V^\rho_{Y}(Z) = \bigcup_{s > 1} V^{\rho,s}$. 

For each $t > 1$, the space $V^{\rho,t}$ is a rational domain of~$Y^\hyb$. Indeed, since $k$ is non-trivially valued, there exists $\alpha \in k$ such that $\abs{\alpha} \in (0,1)$. We then have
\[ V^{\rho,t} := \{x \in Y^\hyb : \abs{\alpha(x)}=1, \ \forall i=1,\dotsc,n, \abs{T_{i}(x)}\le t,\ \forall j=1,\dotsc,m, h_{j}(x) = 0\}.\]
By Lemma~\ref{lem:basisrational}, $V^{\rho,t}$ is over-flow-connected.

Let $t \in \R_{> 1}$. By Corollary~\ref{cor:coherentflowcompact}, the result holds with $V_{Y}^\rho(Z)$ and $V_{Y}^0(Z)$ replaced by $V^{\rho,s}$, for $s\ge t$, and $V^{\rho,t}$ respectively, hence also by $V_{Y}^\rho(Z)$ and $V^{\rho,t}$ respectively.

The result for $V_{Y}^\rho(Z)$ and $V_Y^0(Z)$ follows, using Proposition~\ref{prop:extensioncoherentsheafcompact} and the fact that the family of spaces $(V^{\rho,t})_{t >1}$ forms a basis of neighborhoods of~$V^0_{Y}(Z)$ in~$Y_{0^+}^\hyb$. 

\medbreak

The result for $V_{Y}^r(Z)$ is proved similarly, replacing the sets $V^{\rho,t}$ by the sets 
\[ V^{r,u} := \{x \in Y^\beth : \forall j=1,\dotsc,m, \abs{h_{j}(x)} \le u\}\]
for $u \in (0,1)$.
\end{proof}

We now describe the global sections of the structure sheaf on the spaces $V^\square_{Y}(Z)$. Consider the affine space $\Ahyb{n+m}{k}$ with coordinates $T_{1},\dotsc,T_{n},U_{1},\dotsc,U_{m}$. 

\begin{notation}
Set
\begin{align*}
D_{0}(\infty,0^+) &:= \{x \in (\Ahyb{n+m}{k})_{0^\dag} :  \forall j=1,\dotsc,m, U_{j}(x) = 0\},\\
D_{0}(1^+,1^-) &:= \{x \in (\Ahyb{n+m}{k})_{0^\dag} :  \forall i=1,\dotsc,n, \abs{T_{i}(x)}\le 1,\ \forall j=1,\dotsc,m, \abs{U_{j}(x)} <1\},\\
D_{0}(1^+,0^+) &:= \{x \in (\Ahyb{n+m}{k})_{0^\dag} :  \forall i=1,\dotsc,n, \abs{T_{i}(x)}\le 1,\ \forall j=1,\dotsc,m, U_{j}(x) = 0\},\\
D_{0}(1^+, \,^\ast1^{-}) &:= \{x \in (\Ahyb{n+m}{k})_{0^\dag} :  \forall i=1,\dotsc,n, \abs{T_{i}(x)}\le 1,\ \forall j=1,\dotsc,m, 0 < \abs{U_{j}(x)} <1\}.
\end{align*}
Set also
\[D_{1}(\infty,0^+) = \{x \in (\Ahyb{n+m}{k})_{1} :  \forall j=1,\dotsc,m, U_{j}(x) = 0\}.\]
\end{notation}

We may consider $D_{0}(\infty,0^+)$, $D_{0}(1^+,0^+)$, $D_{0}(1^+,1^-)$ and $D_{0}(1^+, \,^\ast1^{-})$ as subsets of~$\E{n+m}{k_{0}}$. In this case, we denote by $\cO_{k_{0}}$ the overconvergent structure sheaf. Similarly, we may also consider $D_{1}(\infty,0^+)$ as a subset of~$\E{n+m}{\hat k}$, in which case we denote by $\cO_{\hat k}$ the overconvergent structure sheaf.

\begin{lemma}\label{lem:ringk0}
The restriction maps induce isomorphisms 
\[  \cO_{k_{0}}(D_{0}(\infty,0^+)) =  \cO_{k_{0}}(D_{0}(1^+,0^+)) =  \cO_{k_{0}}(D_{0}(1^+,1^-)).\]
Those rings are contained in $k[T_{1},\dotsc,T_{n}][\![U_{1},\dotsc,U_{m}]\!]$.

We have
 \[ \cO_{k_{0}}(D_{0}(1^+,\,^\ast1^-)) =  \cO_{k_{0}}(D_{0}(1^+,1^-))[U_{1}^{-1},\dotsc,U_{m}^{-1}].\]
\end{lemma}
\begin{proof}
Set $T := (T_{1},\dotsc,T_{n})$ and $U:=(U_{1},\dotsc,U_{m})$. For $i=(i_{1},\dotsc,i_{n}) \in \N^n$, set $\abs{i} := i_{1} + \dotsb +i_{n}$ and $T^i := T_{1}^{i_{1}} \dotsb T_{n}^{i_{n}}$. Similarly, for $j=(j_{1},\dotsc,j_{m}) \in \N^m$, set $\abs{j} := j_{1} + \dotsb +j_{m}$ and $U^j := U_{1}^{j_{1}} \dotsb U_{m}^{j_{m}}$. 

The ring $\cO_{k_{0}}(D_{0}(1^+,0^+))$ may be written as the set of power series 
\[\sum_{(i,j) \in \N^{n+m}} a_{i,j}\, T^i U^j \in k[\![T,U]\!]\]
for which there exists $t>1$ and $u>0$ such that 
\[\lim_{(i,j) \to \infty} \abs{a_{i,j}}_{0}\, t^{\abs{i}} u^{\abs{j}} = 0.\]

Let $f \in \cO_{k_{0}}(D_{0}(1^+,0^+))$ and write it as a power series as above. Using the fact that $\abs{a_{i,j}}$ is~0 or~1, we deduce that~$f$ belongs to $k[T][\![U]\!]$.

Let us prove that $f \in \cO_{k_{0}}(D_{0}(\infty,0^+))$. It is enough to prove that, for each $s >1$, we have  $f \in \cO_{k_{0}}(D_{0}(s^+,0^+))$, where 
\[D_{0}(s^+,0^+) := \{x \in (\Ahyb{n+m}{k})_{0} :  \forall i=1,\dotsc,n, \abs{T_{i}(x)}\le s,\ \forall j=1,\dotsc,m, U_{j}(x) = 0\}.\]
Since $t>1$, there exists $e>1$ such that $t^e>s$. We have 
\[\lim_{(i,j) \to \infty} \abs{a_{i,j}}_{0}\, (t^e)^{\abs{i}} (u^e)^{\abs{j}} = \lim_{(i,j) \to \infty} \big( \abs{a_{i,j}}_{0}\, t^{\abs{i}} u^{\abs{j}} \big)^e = 0,\]
hence $f \in \cO_{k_{0}}(D_{0}(s^+,0^+))$.

The other parts of the statement are proven similarly.
\end{proof}

Let us now return to the usual overconvergent structures, inherited from the inclusions into $\Ahyb{n+m}{k}$. We denote the corresponding structure sheaves simply by~$\cO$.

\begin{lemma}\label{lem:D10}
We have natural isomorphisms
\[\cO(D_{0}(\infty,0^+)) = (\cO(D_{0}(1^+,1^-)) = \cO(D_{0}(1^+,0^+)) \simto \cO_{k_{0}}(D_{0}(1^+,0^+)) \cap \cO_{\hat k}(D_{1}(\infty,0^+))\]
and 
\[\cO(D_{0}(1^+,\,^\ast1^-)) = (U_{1},\dotsc,U_{m})^{-1} \, \cO(D_{0}(1^+,1^-)).\]
The Jacobson radical of $\cO(D_{0}(1^+,1^-))$ is generated by $U_{1},\dotsc,U_{m}$.
\end{lemma}
\begin{proof}
Use the same notation as in the proof of Lemma~\ref{lem:ringk0}. By \cite[Corollaire~2.8]{EtudeLocale}, the ring $\cO(D_{0}(1^+,0^+))$ identifies to the set of power series
\[f = \sum_{(i,j) \in \N^{n+m}} a_{i,j}\, T^i U^j \in k[\![T,U]\!]\]
for which there exists $\eps \in (0,1]$, $t>1$ and $u>0$ such that the power series
\[ \sum_{(i,j) \in \N^{n+m}} \norm{a_{i,j}}_{[0,\eps]} \,t^{\abs{i}}  u^{\abs{j}}\]
converges. 

Let $f \in \cO_{k_{0}}(D_{0}(1^+,0^+))$ and write it as above. Recall that
\[ \nm_{[0,\eps]} = \max(\va_{0},\va^\eps) = \max\big(\va_{0}, \sup_{0< \delta \le \eps }(\va^\delta)\big).\]
As a result, the convergence condition splits into several conditions. The one involving~$\va_{0}$ corresponds to the fact that $f$ belongs to $\cO_{k_{0}}(D_{0}(1^+,0^+))$. 

For $\delta \in (0,\eps]$, the condition involving $\va^\delta$ amounts to the convergence of the power series $\sum_{(i,j) \in \N^{n+m}} \abs{a_{i,j}} \,(t^{1/\delta})^{\abs{i}}  (u^{1/\delta})^{\abs{j}}$, that is to say the convergence of $f$ in the neighborhood of $D_{1}((t^{1/\delta})^+,0^+)$ in $(\Ahyb{n+m}{k})_{1}$, where 
\[D_{1}((t^{1/\delta})^+,0^+) := \{x \in (\Ahyb{n+m}{k})_{1} :  \forall i=1,\dotsc,n, \abs{T_{i}(x)}\le t^{1/\delta},\ \forall j=1,\dotsc,m, U_{j}(x) = 0\}.\]
Since $t^{1/\delta}$ goes to~$+\infty$ when $\delta$ goes to~0, $f$ finally belongs to~$\cO_{\hat k}(D_{1}(\infty,0^+))$.

We have proven that $\cO(D_{0}(1^+,0^+)) \subseteq \cO_{k_{0}}(D_{0}(1^+,0^+)) \cap \cO_{\hat k}(D_{1}(\infty,0^+))$. The reverse inclusion is obtained by similar arguments.

The same method allows to compute the global sections of the structure sheaf on the discs of the form 
\[ \{x \in (\Ahyb{n+m}{k})_{0^\dag} :  \forall i=1,\dotsc,n, \abs{T_{i}(x)}\le 1,\ \forall j=1,\dotsc,m, \abs{U_{j}(x)} \le u\},\]
for $u \in [0,1)$, or 
\[ \{x \in (\Ahyb{n+m}{k})_{0^\dag} :  \forall i=1,\dotsc,n, \abs{T_{i}(x)}\le v,\ \forall j=1,\dotsc,m, U_{j}(x) = 0\},\]
for $v > 1$. Since $D_{0}(1^+,1^-)$ and $D_{0}(\infty,0^+)$ are increasing unions of such discs, we may compute the global sections of the structure sheaf on them too.

For $\cO(D_{0}(1^+,\,^\ast1^-))$, we proceed similarly, writing $D_{0}(1^+,\,^\ast1^-)$ as the increasing union of the annuli 
\[ \{x \in (\Ahyb{n+m}{k})_{0^\dag} :  \forall i=1,\dotsc,n, \abs{T_{i}(x)}\le 1,\ \forall j=1,\dotsc,m, s \le\abs{U_{j}(x)} \le t\},\]
for $s \le t \in (0,1)$, and using \cite[Corollaire~1.4.11]{CTCZ} instead of \cite[Corollaire~2.8]{EtudeLocale}.

\medbreak

Let us now prove the final part about the Jacobson radical. Recall that the Jacobson radical of a ring~$B$ may be characterized as the set of elements $b \in B$ such that, for every $c\in B$, $1+cb$ is invertible.

Let $f$ be in the Jacobson radical of $\cO(D_{0}(1^+,0^+))$. By the first part of the statement and Lemma~\ref{lem:ringk0}, $f$ belongs to $k[T_{1},\dotsc,T_{n}][\![U_{1},\dotsc,U_{m}]\!]$ and, for each $c \in k^\ast$, $1+cf$ is invertible in this ring. It follows that $f(T_{1},\dotsc,T_{n},0,\dots,0)=0$. 

Let us write $f = \sum_{(i,j) \in \N^{n+m}} a_{i,j}\, T^i U^j$ as above. For $k=1,\dotsc,m$, set 
\[ g_{k} := \sum_{\stackrel{(i,j) \in \N^{n+m}}{j_{1} = \dotsb = j_{k-1} =0,\ j_{k} \ge 1}} a_{i,j}\, T_{1}^{i_{1}}\dotsb T_{n}^{i_{n}} U_{k}^{i_{k}-1} U_{k+1}^{i_{k+1}} \dotsb U_{m}^{i_{m}} \in k[T_{1},\dotsc,T_{n}][\![U_{1},\dotsc,U_{m}]\!].\]
Since $f(T_{1},\dotsc,T_{n},0,\dots,0)=0$, we have
\[ f = U_{1} \, g_{1} + \dotsb + U_{m} \, g_{m}.\]
By construction, the series $g_{k}$ satisfy the same convergence conditions as~$f$. It follows that they all belong to~$\cO(D_{0}(1^+,1^-))$, hence 
\[ f \in (U_{1},\dotsc,U_{m})\, \cO(D_{0}(1^+,0^+)).\]

Conversely, let $f \in  (U_{1},\dotsc,U_{m})\, \cO(D_{0}(1^+,0^+))$. Let $c\in \cO(D_{0}(1^+,0^+))$. The function $1+cf$ evaluates to~1 at each point of $D_{0}(1^+,0^+)$, hence it does not vanish in a neighbourhood of this set. It follows that it is invertible in $\cO(D_{0}(1^+,0^+))$, as required.
\end{proof}

\begin{theorem}\label{th:VYZStein}
Let $\square \in \{\rho,r,0,\ast\}$. Let~$\cF$ be a coherent sheaf on~$V_{Y}^\square(Z)$. Then, $\cF(V_{Y}^\square(Z))$ is a finitely generated $\cO(V_{Y}^\square(Z))$-module and, for each $q\in \N_{\ge 1}$, we have
\[ H^q(V_{Y}^\square(Z),\cF)=0.\]
Moreover, the global section functor $\cF \mapsto \cF(V_{Y}^\square(Z))$ induces an equivalence between the category of coherent sheaves on~$V_{Y}^\square(Z)$ and the category of finitely generated modules on~$\cO(V_{Y}^\square(Z))$.

We have natural isomorphisms
\[ \cO(D_{0}(1^+,0^+))/\big(J+(U_{1}-h_{1},\dotsc,U_{m}-h_{m})\big) \simto \cO(V_{Y}^0(Z)) = \cO(V^r_{Y}(Z)) = \cO(V^\rho_{Y}(Z))\]
and \[ \cO(D_{0}(1^+,\,^\ast1^-))/\big(J+(U_{1}-h_{1},\dotsc,U_{m}-h_{m})\big) \simto \cO(V_{Y}^\ast(Z)),\]
sending $T_{i}$ to $T_{i}$, for $i \in\{1,\dotsc,n\}$, and $U_{j}$ to $h_{j}$, for $j \in\{1,\dotsc,m\}$. 

In particular, we have
\[\cO(V_{Y}^\ast(Z)) = I^{-1}\, \cO(V^r_{Y}(Z)).\] 
The Jacobson radical of $\cO(V^r_{Y}(Z))$ is generated by~$I$.
\end{theorem}
\begin{proof}
By assumption, $Y$ identifies to the closed subscheme of~$\A^n_{k}$ with coordinates $T_{1},\dotsc,T_{n}$ defined by the ideal~$J$. It follows that $Y^\hyb$ identifies to the closed analytic subspace of~$\Ahyb{n}{k}$ defined by the sheaf of ideals~$\cJ$ generated by~$J$. 

The projection $\pi \colon \Ahyb{n+m}{k} \to \Ahyb{n}{k}$ induces an isomorphism between the closed analytic subset~$E$ of $\Ahyb{n+m}{k} $ defined by~$\pi^*\cJ + (U_{1}-h_{1},\dotsc,U_{m}-h_{m})$ and the closed analytic subset of $\Ahyb{n}{k}$ defined by~$\cJ$ (whose inverse is the morphism defined by $(h_{1},\dotsc,h_{m})$), that it to say~$Y^\hyb$. Under this isomorphism, the subset $V^0_{Y}(Z)$ of~$Y^\hyb$ corresponds to the image of~$D_{0}(1^+,0^+)$ in~$E$. 

Note that $V^0_{Y}(Z)$ is an overconvergent affinoid space, with the terminology of~\cite[D\'efinition~8.3.1]{CTCZ}. Let~$\cF$ be a coherent sheaf on~$V^0_{Y}(Z)$. By \cite[Th\'eor\`eme 8.3.9]{CTCZ}, $\cF(V^0_{Y}(Z))$ is a finitely generated $\cO(V^0_{Y}(Z))$-module, $\cF(V^0_{Y}(Z))$ generates~$\cF$ and the higher cohomology of~$\cF$ vanishes. The first part of the result for~$V^0_{Y}(Z)$ follows, as well as the claimed equivalence of categories.

By \cite[Corollaire 8.2.16]{CTCZ}, coherent sheaves on $D_{0}(1^+,0^+)$ have no higher coherent cohomology. In particular, the first cohomology group of the sheaf of ideals generated by $J+(U_{1}-h_{1},\dotsc,U_{m}-h_{m})$ vanishes. The description of $\cO(V_{Y}^0(Z))$ follows by a standard argument (see for instance the proof of \cite[Th\'eor\`eme~8.3.8]{CTCZ}).

\medbreak

Let us now prove the results for $V^r_{Y}(Z)$. Reasoning as above, we may identify $V^r_{Y}(Z)$ with a closed analytic subset of $D_{0}(1^+,1^-)$. It is enough to prove that the latter has no higher coherent cohomology. For $n\in \N$, set 
\[ D_{n} := \{x \in (\Ahyb{n+m}{k})_{0^\dag} :  \forall i=1,\dotsc,n, \abs{T_{i}(x)}\le 1,\ \forall j=1,\dotsc,m, \abs{U_{j}(x)} \le 1 - 2^{-n}\}.\]
Arguing as in \cite[Exemple~8.5.6]{CTCZ}, we prove that the sequence $(D_{n})_{n\in \N}$ is a universal Stein exhaustion of~$D_{0}(1^+,1^-)$ (see \cite[D\'efinition~8.5.5]{CTCZ}). The result now follows from \cite[Th\'eor\`eme~8.5.14]{CTCZ}.  The finite generation property and the equivalence of categories follow from the analogous results for~$V^0_{Y}(Z)$, together with Proposition~\ref{prop:isoVV0affine}. 

\medbreak

A similar strategy works for $V_{Y}^\rho(Z)$, identifying it to a closed analytic subset of~$D_{0}(\infty,0^+)$ and using the exhaustion by the discs 
\[ \{x \in (\Ahyb{n+m}{k})_{0^\dag} :  \forall i=1,\dotsc,n, \abs{T_{i}(x)}\le n,\ \forall j=1,\dotsc,m, U_{j}(x) = 0\},\]
for $n\in \N$.

\medbreak

Let us now deal with $V_{Y}^\ast(Z)$. To prove the finite generation property and the equivalence of categories, it is enough to find a subspace~$W$ of $V_{Y}^\ast(Z)$ that will play the same role as~$V^0_{Y}(Z)$ for $V^r_{Y}(Z)$. In other words, we are looking for an overconvergent affinoid subspace~$W$ of $V_{Y}^\ast(Z)$ such that $T(W) = V_{Y}^\ast(Z)$. A suitable candidate is given by 
\[W = \{ x \in V_{Y}^\ast(Z) : \forall j \in\{1,\dotsc,m\}, \abs{h_{j}(x)}=s\},\] 
for some $s\in (0,1)$. By \cite[Proposition~8.3.6]{CTCZ}, it is indeed an overconvergent affinoid space, and the rest is clear. 

\medbreak

The final statements about the localization and the Jacobson radical now follow from Lemma~\ref{lem:D10}.
\end{proof}

\begin{corollary}
For each coherent sheaf~$\cF$ on $V_{Y}^\ast(Z)$, there exists a coherent sheaf~$\cG^r$ (resp. $\cG^\rho$) on~$V^r_{Y}(Z)$ (resp. $V^\rho_{Y}(Z)$) such that $\cG^r_{\vert V_{Y}^\ast(Z)} \simeq \cF$ (resp. $\cG^\rho_{\vert V_{Y}^\ast(Z)} \simeq \cF$).
\end{corollary}
\begin{proof}
By Theorem~\ref{th:VYZStein}, the category of coherent sheaves on $V_{Y}^\ast(Z)$ is equivalent to the category of modules of finite type over $\cO(V_{Y}^\ast(Z))$. Similarly, the category of coherent sheaves on $V_{Y}^r(Z)$ (resp. $V_{Y}^\rho(Z)$) is equivalent to the category of modules of finite type over $\cO(V_{Y}^r(Z))$ (resp. $\cO(V^\rho_{Y}(Z))$). The results follows, since $\cO(V_{Y}^\ast(Z))$ is a localization of $\cO(V^r_{Y}(Z))$ (resp. $\cO(V^\rho_{Y}(Z))$). 
\end{proof}

\subsubsection{Extension of coherent sheaves}

In this section, we return to the general setting. Let~$Y$ be a variety over~$k$. Let~$Z$ be an effective Cartier divisor on~$Y$. Let~$\cI$ be the sheaf of ideals defining~$Z$. By assumption, it is locally principal. 

Our goal here is to prove that every coherent sheaf on~$V_{Y}^\ast(Z)$ still extends to a coherent sheaf on~$V_{Y}(Z)$. The statement is reminiscent of results stipulating that coherent sheaves on the generic fiber of a formal scheme may be induced by coherent sheaves on the whole formal scheme, see \cite[Proposition~5.6]{FormalRigidI} or \cite[Proposition~4.8.18]{Proust}. We closely follow the proofs given in those references.

We first extend Proposition~\ref{prop:isoVV0affine}.

\begin{proposition}\label{prop:isoVV0}
For each coherent sheaf~$\cF$ on~$V^\rho_{Y}(Z)$ (resp. $\cF(V_{Y}^r(Z))$), the restriction map
\[ \cF(V^\rho_{Y}(Z)) \too \cF(V_{Y}^0(Z))  \textrm{ (resp. } \cF(V_{Y}^r(Z)) \too \cF(V_{Y}^0(Z))\, ) \]
is an isomorphism.

Moreover, the restriction map induces an equivalence between the categories of coherent sheaves on~$V_{Y}^\rho(Z)$ (resp. $V_{Y}^r(Z)$) and~$V^0_{Y}(Z)$.
\end{proposition}
\begin{proof}
Let $\square \in \{\rho, r\}$. Let $(U_{i})_{i\in I}$ be a finite open affine covering of~$Y$. Let $\cF$ be a coherent sheaf on$V^\square_{Y}(Z)$. By Proposition~\ref{prop:isoVV0affine}, for each $i,j\in I$, the restriction maps
\[ \cF(V^\square_{U_{i}}(Z \cap U_{i})) \too  \cF(V^0_{U_{i}}(Z \cap U_{i}))\]
and 
\[\cF(V^\square_{U_{i}\cap U_{j}}(Z \cap U_{i}\cap U_{j})) \too  \cF(V^0_{U_{i}\cap U_{j}}(Z \cap U_{i}\cap U_{j}))\]
are isomorphisms, for each $i,j \in I$.

By Lemma~\ref{lem:unionintersection}, the family $(V^\square_{U_{i}}(Z \cap U_{i}))_{i\in I}$ (resp. $(V^0_{U_{i}}(Z \cap U_{i}))_{i\in I}$) covers $V^\square_{Y}(Z)$ (resp. $V^0_{Y}(Z)$) and, for $i,j \in I$, we have $V^\square_{U_{i} \cap U_{j}}(Z\cap U_{i}\cap U_{j}) = V^\square_{U_{i}}(Z\cap U_{i}) \cap V^\square_{U_{j}}(Z\cap U_{j})$ (resp. $V^0_{U_{i} \cap U_{j}}(Z\cap U_{i}\cap U_{j}) = V^0_{U_{i}}(Z\cap U_{i}) \cap V^0_{U_{j}}(Z\cap U_{j})$). The first result follows. The second one is proved similarly.
\end{proof}

\begin{lemma}\label{lem:extensionsubsheaf}
Let $\square \in \{\rho,r,0\}$. Let $U$ be an open subset of~$Y$ and set $V^\square_{U}(Z) := V^\square_{U}(Z\cap U)$. Let~$\cF$ be a coherent sheaf on~$V^\square_{Y}(Z)$ and let~$\cG$ be a coherent subsheaf of~$\cF_{\vert V^\square_{U}(Z)}$. Then, there exists a coherent subsheaf~$\cG'$ of~$\cF$ such that $\cG'_{\vert V^\square_{U}(Z)} = \cG$. 
\end{lemma}
\begin{proof}
By Proposition~\ref{prop:isoVV0}, it is enough to prove the statement for $\square = \rho$. The analytification $Z^\an_{0}$ of~$Z$ over~$k_{0}$ may naturally be seen as a subset of the analytification~$Z^\hyb$ of~$Z$ over~$k_{\hyb}$. Recall that we denote it by~$Z^\hyb_{0^\dag}$, when endowed with the overconvergent structure sheaf inherited from~$Z^\hyb$ (see Notation~\ref{nota:Xhyb}).

Note that $V^\rho_{Y}(Z)$ and $Z^\hyb_{0^\dag}$ have the same underlying topological spaces. More precisely, we have a closed immersion
\[ j \colon Z^\hyb_{0^\dag} \too V^\rho_{Y}(Z),\]
which is defined by the sheaf of ideals~$\cI^{\hyb}$. It restricts to a closed immersion
\[ j_{U} \colon (Z\cap U)^\hyb_{0^\dag} \too V^\rho_{U}(Z).\]

By Theorem~\ref{th:GAGAhybrid}, we have an equivalence between the categories of coherent sheaves on~$Z$ (resp. $Z\cap U$) and~$Z^\hyb_{0^\dag}$ (resp. $(Z\cap U)^\hyb_{0^\dag}$). By \cite[Th\'eor\`eme~6.9.7]{EGAInew}, given a coherent sheaf on~$Z$, any coherent subsheaf of its restriction to~$Z\cap U$ extends to a coherent subsheaf on the whole~$Z$. It follows that there exists a coherent subsheaf $\cH$ of $j^\ast  \cF$ such that $\cH_{\vert  (Z\cap U)^\hyb_{0^\dag}} = j_{U}\!^\ast \cG$. 

Set $\cG' := \sKer(\cF \to j_{\ast} (j^\ast \cF/\cH))$. It is a coherent subsheaf of~$\cF$ and we claim that $\cG'_{\vert V^\rho_{U}(Z)} = \cG$. This is a local statement, so, using Lemma~\ref{lem:unionintersection}, we may assume that~$U$ is affine. By Theorem~\ref{th:VYZStein}, it is enough to show that the modules of global sections~$M'$ of $\cG'_{\vert V^\rho_{U}(Z)}$ and~$M$ of~$\cG$ are equal. By Theorem~\ref{th:VYZStein}, $M$ and~$M'$ are finitely generated over~$\cO(V^\rho_{U}(Z))$ and, by construction, we have 
\[j_{U}\!^\ast \cG'_{\vert V_{U}^\rho(Z)} = j_{U}\!^\ast \cG,\]
hence $M'/\cI^\hyb(V^\rho_{U}(Z)) = M/\cI^\hyb(V^\rho_{U}(Z))$. By Theorem~\ref{th:VYZStein}, $\cI^\hyb(V^\rho_{U}(Z))$ is the Jacobson radical of $\cO(V^\rho_{U}(Z))$, so that the statement is now a consequence of Nakayama's lemma.
\end{proof}

\begin{lemma}\label{lem:gluingU1U2}
Let $U_{1}$ and $U_{2}$ be quasi-compact open subsets of~$Y$. For $i\in \{1,2\}$, set $V^r_{i} := V^r_{U_{i}}(Z \cap U_{i})$ and $V^\ast_{i} := V_{U_{i}}^\ast(Z \cap U_{i})$. Let $\cF_{1}$ and $\cF_{2}$ be coherent sheaves on~$V_{1}^\ast$ and~$V_{2}^\ast$ respectively. Assume that they extend to coherent sheaves on~$V^r_{1}$ and~$V^r_{2}$ respectively, and that there exists an isomorphism
\[ v \colon (\cF_{1})_{\vert V^\ast_{1}\cap V^\ast_{2}} \simto (\cF_{2})_{\vert V^\ast_{1}\cap V^\ast_{2}}.\]
Then there exist a coherent sheaf~$\cF$ on~$V^r_{1}\cup V^r_{2}$ and isomorphisms $\cF_{\vert V_{1}^\ast} \simeq \cF_{1}$ and $\cF_{\vert V_{2}^\ast} \simeq \cF_{2}$ that are compatible with~$v$.
\end{lemma}
\begin{proof}
For $i\in \{1,2\}$, let $\cG_{i}$ be a coherent sheaf on~$V^r_{i}$ such that $(\cG_{i})_{\vert V_{i}^\ast} = \cF_{i}$. Moding out by the annihilator of~$\cI$, which is coherent, we may assume that $\cG_{1}$ and~$\cG_{2}$ have no $\cI$-torsion. It then follows from Theorem~\ref{th:VYZStein} that, for each $i\in \{1,2\}$ and each open subset~$W$ of~$V^r_{i}$, the restriction map 
\[ \rho_{i} \colon \cG_{i}(W) \to \cG_{i}(W\cap V_{i}^\ast) = \cF_{i}(W\cap V_{i}^\ast)\] 
is injective.

We claim that there exists an integer $n\in \N$ such that, for each open subset~$W$ of~$V^r_{1}\cap V^r_{2}$, we have
\[ v(\rho_{1}(\cI^n \cG_{1}(W))) \subseteq \rho_{2}(\cG_{2}(W)).\]

Since $Y$ is separated, $U_{1}\cap U_{2}$ admits a finite covering by open affine subsets~$(U'_{j})_{j\in J}$. We may assume that, for each $j\in J$, $\cI_{\vert U'_{j}}$ is principal, say generated by $f_{j} \in \cO(U'_{j})$. By Lemma~\ref{lem:unionintersection}, the family $(V^r_{U'_{j}}(Z \cap U'_{j}))_{j\in J}$ forms an open cover of $V^r_{1}\cap V^r_{2}$ and, by Theorem~\ref{th:VYZStein}, it is enough to prove the result for global sections over each $V^r_{U'_{j}}(Z \cap U'_{j})$ (with an integer~$n$ that may depend on~$j$).

Let $j\in J$. By Theorem~\ref{th:VYZStein}, $\cG_{1}(V^r_{U'_{j}}(Z \cap U'_{j}))$ is finitely generated over $\cO(V^r_{U'_{j}}(Z \cap U'_{j}))$, so it is enough to prove that, for each $s\in \cG_{1}(V^r_{U'_{j}}(Z \cap U'_{j}))$, there exists $n\in \N$ such that 
\[v(\rho_{1}(f_{j}^n s)) = f_{j}^n v(\rho_{1}(s)) \in \rho_{2}(\cG_{2}(V^r_{U'_{j}}(Z \cap U'_{j}))).\]
By Theorem~\ref{th:VYZStein} again, the restriction map 
\[\rho_{2} \colon \cG_{2}(V^r_{U'_{j}}(Z \cap U'_{j})) \too \cG_{2}(V^\ast_{U'_{j}}(Z \cap U'_{j}))\]
is nothing but the localization by~$f_{j}$. Since $v(\rho_{1}(s))$ belongs to its image, the result follows.

\medbreak

Using the preceding result, we may identify $(\cI^n \cG_{1})_{\vert V^r_{1}\cap V^r_{2}}$ to a subsheaf of~$(\cG_{2})_{\vert V^r_{1}\cap V^r_{2}}$. By Lemma~\ref{lem:extensionsubsheaf}, $(\cI^n \cG_{1})_{\vert V^r_{1}\cap V^r_{2}}$~extends to a subsheaf~$\cG'_{2}$ of~$\cG_{2}$. Gluing $\cI^n \cG_{1}$ and $\cI^n\cG_{2} + \cG'_{2}$ along $(\cI^n \cG_{1})_{\vert V^r_{1}\cap V^r_{2}}$ over $V^r_{1}\cap V^r_{2}$, we get a coherent sheaf~$\cF$ that satisfies the properties of the statement.
\end{proof}

\begin{theorem}\label{th:extension}
Every coherent sheaf on~$V_{Y}^\ast(Z)$ extends to a coherent sheaf on~$V^r_{Y}(Z)$.
\end{theorem}
\begin{proof}
Let $\cF$ be a coherent sheaf on~$V_{Y}^\ast(Z)$. Let $(U_{i})_{i\in I}$ be a finite open affine covering of~$Y$. We may assume that, for each $i\in I$, $\cI_{\vert U_{i}}$ is principal, say generated by $f_{i} \in \cO(U_{i})$.  By Lemma~\ref{lem:unionintersection}, the families $(V^r_{i} := V^r_{U_{i}}(Z \cap U_{i}))_{i\in I}$ and $(V_{i}^\ast := V_{U_{i}}^\ast(Z \cap U_{i}))_{i\in I}$ form open covers of $V^r_{Y}(Z)$ and $V_{Y}^\ast(Z)$ respectively.

Let $i\in I$. It follows from Theorem~\ref{th:VYZStein} that there exists an exact sequence
\[ (\cO_{V_{i}^\ast})^{n_{i}} \xrightarrow[]{u_{i}}  (\cO_{V_{i}^\ast})^{m_{i}} \too \cF_{\vert V_{i}^\ast} \too 0.\]
Let $(e_{i,1},\dotsc,e_{i,n_{i}})$ denote the canonical basis of the $\cO(V_{i}^\ast)$-module $\cO(V_{i}^\ast)^{n_{i}}$. By Theorem~\ref{th:VYZStein} again, there exists $N_{i}\in \N$ such that, for each $j \in \{1,\dotsc,n_{i}\}$, we have $f_{i}^{N_{i}} u_{i}(e_{i,j}) \in \cO(V^r_{i})^{m_{i}}$. The cokernel~$\cG_{i}$ of the morphism
\[{\renewcommand{\arraystretch}{1.3}\begin{array}{ccc}
(\cI_{V^r_{i}})^{n_{i}} & \too &  (\cO_{V^r_{i}})^{m_{i}}\\
f_{i}^{N_{i}}\, e_{i,j} & \mapstoo & f_{i}^{N_{i}} u_{i}(e_{i,j})
\end{array}}\]
is then a coherent sheaf on~$V^r_{i}$ whose restriction to~$V_{i}^\ast$ is isomorphic to~$\cF_{\vert V_{i}^\ast}$. 

By repeatedly using Lemma~\ref{lem:gluingU1U2}, we may now glue the sheaves~$\cG_{i}$ into a coherent sheaf~$\cG$ on~$V^r_{Y}(Z)$ whose restriction to~$V_{Y}^\ast(Z)$ is isomorphic to~$\cF$.
\end{proof}

\subsection{Back to compactifications}

We now apply the extension result from Theorem~\ref{th:extension} to study coherent sheaves on valuative compactifications.

\begin{proposition}\label{prop:s=0}
Let $U$ be an analytic space over~$k_{\hyb}$. For each coherent sheaf~$F$ on~$U$ and each $s \in F(U)$, if $s_{\vert U_{>0}} =0$, then $s=0$. 
\end{proposition}
\begin{proof}
Let $F$ be a coherent sheaf on~$U$ and let $s\in F(U)$. Assume that $s_{\vert U_{>0}} =0$. It is enough to prove that, for each $x \in U_{0}$, the image $s_{x}$ of~$s$ in~$F_{x}$ is~0. 

Let $x \in U_{0}$. There exists an open neighborhood~$V$ of~$x$ in~$U$ such that $F_{\vert V}$ is isomorphic to a quotient $\cO_{V}^n/M$, where $M$ is a coherent submodule of~$\cO_{V}^n$. Up to shrinking~$V$, we may assume that $s_{\vert V}$ lifts to an element $t \in \cO(V)^n$. Set $G := (t\,\cO_{V}^n + M)/M$. It is a coherent sheaf on~$V$ and, by assumption, we have $G_{\vert V_{>0}} = 0$. 

Assume by contradiction that $G_{x} \ne 0$. Since the support~$\Supp(G)$ of~$G$ identifies to the sheaf of ideals~$\Ann_{\cO_{V}}(G)$, it is a closed analytic subset of~$V$. In particular, $\Supp(G)$ is naturally endowed with a structure of analytic space over $k_{\hyb}$. By \cite[Proposition~6.4.1]{CTCZ}, the projection map $\Supp(G) \to \cM(k_{\hyb})$ is open at~$x$, which contradicts the fact that $G_{\vert V_{>0}} = 0$. It follows that $G_{x} =0$, that it to say $t_{x} \in M_{x}$, and $s_{x}=0$.
\end{proof}

\begin{corollary}\label{cor:injection>0}
Let $U$ be an analytic space over~$k_{\hyb}$. Let $i_{>0} \colon U_{>0} \to U$ denote the open embedding of~$U_{>0}$ into~$U$. For each coherent sheaf~$F$ on~$U$, the canonical morphism $F \to (i_{>0})_{\ast} F_{\vert U_{>0}}$ is injective.

In particular, for each variety $X$ over~$k$ and each coherent sheaf~$\cF$ on~$X$, the restriction map
\[ H^0(X^+,\cF^\hyb) \too H^0(X^\an, \cF^\an)\]
is injective.
\end{corollary}
\begin{proof}
Since each point in~$U_{0}$ has a basis of neighborhoods consisting of analytic spaces over~$k_{\hyb}$, the first part of the statement follows from Proposition~\ref{prop:s=0}.

Let us now prove the last part of the statement. Let $X$ be a variety over~$k$ and $\cF$ be a coherent sheaf on~$X$. It follows from the first part that the restriction map $H^0(X^+,\cF^\hyb) \to H^0(X^+_{>0}, \cF^\hyb)$ is injective. The results now follows from Lemmas~\ref{lem:hybaneps} and~\ref{lem:>0eps}, applied with $\eps=1$.

\end{proof}

\begin{corollary}\label{cor:iso>0}
Let $U$ be an analytic space over~$k_{\hyb}$. Let $F$ and $G$ be coherent sheaves on~$U$. Then, $F$ and $G$ are isomorphic as $\cO_{U}$-modules if, and only if, $F_{\vert U_{>0}}$ and $G_{\vert U_{>0}}$ are isomorphic as $\cO_{U_{>0}}$-modules .\end{corollary}
\begin{proof}
Only the converse statement requires a proof. Assume that there exists an isomorphism of $\cO_{U_{>0}}$-modules $\varphi \colon F_{\vert U_{>0}} \simto G_{\vert U_{>0}}$.

Let $i_{>0} \colon U_{>0} \to U$ denote the open embedding of~$U_{>0}$ into~$U$. By Corollary~\ref{cor:injection>0}, we may identify $F$ (resp. $G$) to a subsheaf of $(i_{>0})_{\ast} F_{\vert U_{>0}}$ (resp. $(i_{>0})_{\ast} G_{\vert U_{>0}}$). Using $\varphi$, we may identify $F$ and $G$ to subsheaves of the same sheaf $(i_{>0})_{\ast} G_{\vert U_{>0}}$. 

Since $F$ and $G$ are coherent, so is the sheaf $(F + G)/F$. By assumption, its support is contained in~$X^+_{0}$, hence $(F + G)/F = 0$, by Proposition~\ref{prop:s=0}. We deduce that $G \subseteq F$. Similarly, we prove that $F \subseteq G$, hence $F = G$. 
\end{proof}

\begin{lemma}\label{lem:globalsectionsmeromorphic}
Let $X$ be a variety over~$k$. Let $\cF$ be a coherent sheaf on~$X$. Let $Y$ be a compactification of~$X$ such that $Z := Y-X$ is a Cartier divisor. Let~$\cG$ be a coherent sheaf on~$Y$ extending~$\cF$. Then the image of the restriction map
\[ H^0(X^+,\cF^\hyb) \too H^0(X^\an, \cF^\an)\]
lies in $H^0(Y^\an, \cG^\an[\ast Z^\an])$.\footnote{We denote by $\cG^\an[\ast Z^\an]$ the sheaf of sections of~$\cG^\an$ meromorphic along~$Z^\an$, see \cite[footnote 5]{GrothendieckdeRham}.}
\end{lemma}
\begin{proof}
Let $s \in H^0(X^+,\cF^\hyb)$. It is enough to show that each $z \in Z^\an$ admits a neighborhood~$U_{z}$ in $Y^\an$ such that $s_{\vert U_{z}}$ belongs to $H^0(U_{z}, \cG^\an[\ast Z^\an])$.

Let $z \in Z^\an$. Let $U$ be an open affine neighborhood of~$z$ in~$Y$ such that $Z\cap U$ is defined by the vanishing of a single element $h$ in $\cO(U)$. Write $\cO(U) = k[T_{1},\dotsc,T_{n}]/J$. We may assume that, for each $i \in \{1,\dotsc,n\}$, we have $\abs{T_{i}(z)} < 1$. 

Since~$U$ is affine, there exists a surjective morphism $\cO_{U}^{N} \to \cG_{\vert U}$. It restricts to a surjective morphism $\cO_{X \cap U}^{N} \to \cF_{\vert X\cap U}$ and, by Theorem~\ref{th:VYZStein}, to a surjective morphism
\[H^0(V_{U}^\ast(Z \cap U),\cO)^{N} \too H^0(V_{U}^\ast(Z\cap U),\cF^\hyb).\]
By the final part of Theorem~\ref{th:VYZStein}, there exists $m\in \N$ such that $h^{m} s$ lies in the image of $H^0(V_{U}^r(Z),\cO)^{N}$, that is to say $h^m s$ belongs to $H^0(V_{U}^r(Z),\cG^\hyb)$. Recall that $V_{U}^r(Z)$ is endowed with the overconvergent structure inherited from~$Y^\hyb$, hence the result holds in some neighborhood of $V_{U}^r(Z)$ in~$Y^\hyb$. Using Corollary~\ref{cor:coherentflowcompact} and the fact that $\abs{T_{i}(z)} < 1$ for all i, we deduce that it holds in the neighborhood of~$z$.
\end{proof}

\begin{theorem}\label{th:H0}
Let $X$ be a variety over~$k$. Let $\cF$ be a coherent sheaf on~$X$. The analytification map $H^0(X,\cF) \to H^0(X^\hyb,\cF^\hyb)$ induces an isomorphism
\[ H^0(X,\cF) \simto H^0(X^+,\cF^\hyb) = H^0(X\cp, \cF\cp).\]
\end{theorem}
\begin{proof}
Consider the morphism $\psi \colon H^0(X^+,\cF^\hyb) \too H^0(X^\an, \cF^\an)$ induced by the restriction. By Proposition~\ref{prop:s=0}, it is injective.

Let $Y$ be a compactification of~$X$ such that $Z := Y-X$ is a Cartier divisor. By \cite[Corollaire~6.9.11]{EGAInew}, there exists a coherent sheaf~$\cG$ on~$Y$ extending~$\cF$. By GAGA, we have an isomorphism $H^0(Y^\an, \cG^\an[\ast Z^\an]) \simeq H^0(Y, \cG[\ast Z])  \simeq H^0(X,\cF)$, hence, by Lemma~\ref{lem:globalsectionsmeromorphic}, the image of~$\psi$ lies in $H^0(X,\cF)$. 

To conclude, it is enough to remark that the morphism $H^0(X,\cF) \to H^0(X^+,\cF^\hyb)$ induced by the analytification is a section of~$\psi$.
\end{proof}

Applying the result to the structure sheaf and counting idempotents, we obtain the following result.

\begin{corollary}
Let $X$ be a variety over~$k$. We have canonical bijections
\[ \pi_{0}(X) = \pi_{0}(X^+) = \pi_{0}(X\cp).\]
\qed
\end{corollary}

\begin{remark}
The result of Theorem~\ref{th:H0} does not extend to higher cohomology groups in general. Indeed, we have $H^2(\A^{2,+}_{k},\cO) \ne 0$. 

Let us prove this fact. By Example~\ref{ex:A2}, we have 
\[ (\A^2_{k})_{\infty} = \{  x \in \E{2}{k_{0}} : \max(\abs{T_{1}(x),\abs{T_{2}(x)}}) > 1\},\]
so we may write $X=U_{0} \cup U_{1} \cup U_{2}$, with $U_{0} := (\Ahyb{2}{k})_{>0}$ and, for $i=1,2$,
\[ U_{i} := \{ x\in \Ahyb{2}{k} : \abs{T_{i}(x)} >1\}.\]
By \cite[Corollaire~8.5.24]{CTCZ}, the covering $\{U_{0},U_{1},U_{2}\}$ of~$\A^{2,+}_{k}$ is acyclic for any coherent sheaf, which allows to compute cohomology \textit{via} \v Cech cohomology. In particular, $H^2(\A^{2,+}_{k},\cO)$ identifies to the cokernel of the morphism
\[  \varphi \colon H^0(U_{0}\cap U_{1},\cO) \oplus H^0(U_{0}\cap U_{2},\cO) \oplus H^0(U_{1}\cap U_{2},\cO) \too H^0(U_{0}\cap U_{1}\cap U_{2},\cO),\]
given by the alternate sum of the restrictions.

Arguing as in the proof of Lemma~\ref{lem:D10}, we may write elements of $H^0(U_{0}\cap U_{1}\cap U_{2},\cO)$ and the various $H^0(U_{i}\cap U_{j},\cO)$'s as series of the form $\sum_{u,v\in \Z} a_{u,v} T_{1}^u T_{2}^v$ satisfying certain conditions. Let us focus on the non-negative parts of the expansions. For elements of $H^0(U_{0}\cap U_{1},\cO)$ (resp. $H^0(U_{0}\cap U_{2},\cO)$, resp. $H^0(U_{1}\cap U_{2},\cO)$), they involve only finitely many powers of~$T_{1}$ (resp. $T_{2}$, resp. $T_{1}$ and $T_{2}$). On the other hand, any power series that converges on the whole~$\E{2}{\hat k}$ gives rise to an element of $H^0(U_{0}\cap U_{1}\cap U_{2},\cO)$. We deduce that $\coker(\varphi)$, hence $H^2(\A^{2,+}_{k},\cO)$, is non-zero, and even infinite-dimensional.
\end{remark}

%

To conclude, we compare the categories of coherent sheaves on~$X$ and~$X^+$ or~$X\cp$.


\begin{theorem}\label{th:FF+}
Let $X$ be a variety over~$k$. The functor $\cF \mapsto \cF^\hyb_{\vert X^+}$ (resp. $\cF \mapsto \cF\cp$) from the category of coherent sheaves on~$X$ the category of coherent sheaves on~$X^+$ (resp. $X\cp$) is fully faithfull.

If $k$ is complete, then it is an equivalence of categories. 
\end{theorem}
\begin{proof}
The full faithfullness of the functor follows from Theorem~\ref{th:H0} applied to $\sHom$ sheaves. 

\medbreak

Let us now assume that $k$ is complete and prove that the functor is essentially surjective. By Corollary~\ref{cor:eqX+Xcp}, it is enough to handle the case of~$X^+$.

Let us fix a compactification~$Y$ of~$X$. Set $Z := Y-X$. By blowing-up, we may assume that $Z$ is a Cartier divisor.

Let~$F$ be a coherent sheaf on $X^+ = X_{\infty} \cup X^\hyb_{>0}$. With Notation~\ref{nota:VYZ}, we have $X_{\infty} = V^\ast_{Y}(Z)$. By Theorem~\ref{th:extension}, $F_{\vert V^\ast_{Y}(Z)}$ extends to a coherent sheaf~$G$ on~$V^r_{Y}(Z) = r^{-1}_{Y}(Z)$. 

Let $(Y_i)_{i\in I}$ be a finite affine covering of~$Y$ such that, for each $i\in I$, the ideal of the closed subscheme $Z\cap Y_{i}$ of~$Y_{i}$ is generated by a single function, say $h_{i} \in \cO(Y_{i})$. 

For each $i\in I$, let $(f_{i,j})_{j \in J_{i}}$ be a finite family of generators of the $k$-algebra $\cO(Y_{i})$. For $n \in \N$, set 
\[V_{i,n} := \{ x \in Y_{i}^\hyb : \forall j \in J_{i}, \abs{f_{i,j}} \le n.\}\]
The family $(V_{i,n})_{n\in \N}$ is a compact exhaustion of~$Y_{i}^\hyb$.

Since $Y^\hyb$ is compact, it is covered by a family of the form $(V_{i,n})_{i\in I}$, for some $n\in \N$.

Fix $r\in (0,1)$. For $i\in I$ and $\eps \in (0,1]$, set
\[W^\le_{i,\eps} :=  \{ x\in V_{i,n+1} \cap Y^\an_{i,[0,\eps]} : \abs{h_{i}} \le r\}\] 
and
\[W^\ge_{i,\eps} :=  \{ x\in V_{i,n+1} \cap Y^\an_{i,[0,\eps]} : \abs{h_{i}} \ge r\} =  \{ x\in V_{i,n+1} \cap X^\an_{i,[0,\eps]} : \abs{h_{i}} \ge r\}.\]

Let $i\in I$. The sheaf~$G$ is defined on $r^{-1}_{Y}(Z)$, hence on $W^\le_{i,0}$. The latter being compact, by Proposition~\ref{prop:extensioncoherentsheafcompact}, $G$ extends to some neighborhood of it, hence to $W^\le_{i,\eps_{i}}$, for some $\eps_{i} \in (0,1]$. We still denote the extension by~$G$. 

Set $\eps := \min_{i\in I}(\eps_{i}) \in \R_{>0}$. Set $W^\le := \bigcup_{i\in I} W^\le_{i,\eps}$, $W^\ge := \bigcup_{i\in I} W^\ge_{i,\eps}$, $G' := G_{\vert W^\le}$ and $F' := F_{\vert X^+ \cap W^\ge}$. Since $F$ and~$G$ are isomorphic on~$X_{\infty}$, $F'$ and~$G'$ are isomorphic on~$W^\le \cap W^\ge \cap Y^\an_{0}$. The latter isomorphism extends to a neighborhood of $W^\le \cap W^\ge \cap Y^\an_{0}$, hence to $W^\le \cap W^\ge \cap Y^\an_{[0,\eps']}$ for some $\eps' \in (0,\eps]$. As a result, we may glue~$F'$ and~$G'$ on $W^\le \cup (X^+ \cap W^\ge) \cap Y^\an_{[0,\eps']}$. Let us denote the resulting sheaf by~$H$. 

Let $i\in I$. Let $s \in (r,1)$. By construction, $H$ is isomorphic to~$F$ on the open set 
\[ U_{i} := \{ x \in Y_{i}^\hyb \cap \pr^{-1}((0,\eps')) : \forall j \in J_{i}, \abs{f_{i,j}} < n+1, \abs{h_{i}} > s\}.\]

By Theorem~\ref{th:coherentflow}, the isomorphism extends to~$T(U_{i})$, which contains 
\[ \{ x \in  Y_{i}^\hyb \cap \pr^{-1}((0,\eps'))  : \forall j \in J_{i}, \abs{f_{i,j}} < n+1, \abs{h_{i}} > 0\}.\]

For different choices of~$i$, the extended isomorphisms coincide on the common domains of definition, as seen by applying Theorem~\ref{th:coherentflow} on the intersections. Using the fact that the $V_{i,n}$'s cover $Y^\hyb$, we deduce that $H$ is isomorphic to~$F$ on $X^\hyb \cap \pr^{-1}((0,\eps'))$.

Let $\eps'' \in (0,\eps')$. Note that $H$ is defined on the whole~$Y^\an_{\eps''}$, hence, by GAGA, 
there exists a coherent sheaf~$\cH$ on $Y_{\hat k_{\eps''}} = Y_{k}$ such that~$\cH^\hyb$ is isomorphic to~$H$ on $Y^\an_{\eps''}$.

We have shown above that $F$ is isomorphic to $H$ on $X^\an_{\eps''}$, hence to $\cH^\hyb$. By Lemmas~\ref{lem:hybaneps} and~\ref{lem:>0eps}, it follows that $F$ is isomorphic to $\cH^\hyb$ on $X^\hyb_{>0}$, hence on~$X^+$, by Corollary~\ref{cor:iso>0}. This proves the essential surjectivity of the functor $\cF \mapsto \cF^\hyb_{\vert X^+}$.
\end{proof}

%
%

\begin{remark}
The essential surjectivity of the functor stated in Theorem~\ref{th:FF+} fails when $k$ is not complete. Indeed, remark that it implies, \textit{via} the relative spectrum construction, an equivalence of categories between the finite covers of a variety~$X$ and that of~$X\cp$. However, any finite cover of~$X^\an_{\hat k}$ that is trivial outside a compact extends to a finite cover of~$X\cp$, but all such covers do not come from covers of~$X$ (over~$k$). 

\end{remark}

\bibliography{biblio}

\bibliographystyle{amsalpha}

\end{document}